\documentclass[12pt]{amsart}
\usepackage{amsmath,amsfonts,latexsym,graphicx,amssymb,url}
\usepackage{amsmath,txfonts,pifont,bbding,pxfonts,manfnt}
\usepackage{psfrag}
\usepackage{wrapfig, subfigure,graphicx}
\usepackage{hyperref}
\usepackage[active]{srcltx}
\usepackage{pdfsync}
\usepackage{amssymb}
\usepackage{amsthm}
\usepackage{amsmath}
\usepackage{graphicx}
\usepackage{color}
\usepackage{mathtools}
\newcommand{\downineq}{\rotatebox{90}{$\leq $}}
\newcommand{\verteq}{\rotatebox{90}{$= $}}
\allowdisplaybreaks
\numberwithin{equation}{section}
\newcommand{\R}{{\mathbb R}} 

\def \r {{\mathbf r}}

\def \e {\varepsilon}

\def \p {\partial}

\def \e {\mathbf e}
\def\t{\mathcal T}
\def\p{\rho}
\def\m{\bar \Omega}
\def\M{\bar \Omega^{*}}
\def\r{\mathcal R}
\def\n{\mathcal{N}}
\def\b{\textbf{b}}

\theoremstyle{plain}
\newtheorem{theorem}{Theorem}[section]

\newtheorem{lemma}[theorem]{Lemma}
\newtheorem{proposition}[theorem]{Proposition}
\newtheorem{definition}[theorem]{Definition}
\newtheorem{remark}[theorem]{Remark}

\providecommand{\bysame}{\makebox[3em]{\hrulefill}\thinspace}

\begin{document}

\title{On the Numerical Solution of the Far Field Refractor Problem}
\author{Roberto De Leo, Cristian E. Guti\'errez and Henok Mawi} 
\thanks{\today\\
C. E. G was partially supported
by NSF grant DMS--1201401.\\The numerical part in Section \ref{sec:numericalanalysis} was developed first on the computational cluster of the Italian National Institute for Nuclear Physics (INFN), and lately on the cluster of the College of Arts and Sciences at Howard University.} 
\address{Department of Mathematics\\Howard University\\Washington, D.C. 20059}
\email{roberto.deleo@howard.edu, henok.mawi@howard.edu}
\address{Department of Mathematics\\Temple University\\Philadelphia, PA 19122}
\email{gutierre@temple.edu}

\maketitle 

\begin{abstract}
The far field refractor problem with a discrete target is solved with a numerical scheme that uses and simplify ideas from \cite{Caf-K-Oli:antenna}.
A numerical implementation is carried out and examples are shown.
\end{abstract}

\section{Introduction}
The purpose of this paper is to present an algorithm to construct far field one source refractors with arbitrary precision.
We use the ideas from the paper \cite{Caf-K-Oli:antenna} by Caffarelli, Kochengin and Oliker, where they develop an algorithm to construct far field point source global reflectors, i.e., the source domain $\Omega$ is the whole sphere $S^2$, and the density is smooth.
For our refraction problem, we are able to simplify and extend these ideas to deal with densities that are only bounded and work in general domains. In particular, we do not need to consider derivatives of the refractor measure, we only need to prove an appropriate Lipschitz bound for the refractor measure which considerably simplifies the approach proposed in  \cite{Caf-K-Oli:antenna}. 
In addition, our approach does not use the mass transport structure of the
far field problem, and therefore it can be used in near field problems.
Since we are working in general domains $\Omega$ and with a non smooth density, the differentiability of the refractor measure might not hold in general. This depends on the shape and regularity of the domain and the smoothness of the density. 
The nature of refraction problems demands for domains for which total internal reflection does not occur, see condition \eqref{geometricconstraintfork<1}. Therefore the global problem does not make sense in this case.

To place our results in perspective we mention the following. Recently, Castro, M\'erigot and Thibert 
\cite{Castro2015}
introduced numerical methods to solve the reflector problem. These are based on optimal transport ideas introducing a concave function arising from the Kantorovitch functional. This function is analyzed numerically and their results follow, combined with other numerical packages. An advantage of this approach is that the convergence of their algorithm is faster than the one proposed in \cite{Caf-K-Oli:antenna}.  For general cost functions satisfying the Ma, Trudinger and Wang condition arising in optimal transport \cite{MaTrudingerWang:regularityofpotentials}, the algorithm in \cite{Caf-K-Oli:antenna} is extended in \cite{kitagawa:iterativeschemeoptimaltransport} when the density is $C^\infty$ and the domains are convex with respect to the cost function. 
We remark that this does not include our results when the density is smooth, since the refractor considered in the present paper is for $\kappa<1$ and the condition of Ma, Trudinger and Wang does not hold in this case; see \cite[Section 5]{gutierrez-huang:farfieldrefractor}. 
We believe the case $\kappa<1$ is more interesting for lens design since lenses are made of materials that are denser than the surrounding medium. In fact, if the material around the source is cut out with sphere centered at the source, then the lens sandwiched  between that sphere and the constructed refractor surface will perform the desired refracting job.

The far field refractor problem has been considered and solved for the first time in \cite{gutierrez-huang:farfieldrefractor} using optimal mass transport. Several models and variants have been introduced to reflect more accurately the physical features of the problem; see
\cite{gutierrez-mawi:refractorwithlossofenergy}, \cite{gutierrez-tournier:parallelrefractor}, \cite{gutierrez-tournier:REGULARITYFORTHENEARFIELDPARALLELREFRACTORANDREFLECTORPROBLEMS}, and \cite{gutierrez-sabra:thereflectorproblemandtheinversesquarelaw}.
For numerical results to design reflectors solving Monge-Amp\`ere type pdes we refer to \cite{MAforreflectorsnumericalmethods} and 
\cite{platen:MAforREFRACTORSnumericalmethods} both containing many references. 

The organization of the paper is as follows. In Section \ref{sec:preliminaries}, we explain the set up and the problem solved. In Section \ref{subset:lemmasfortracingmapandrefractormeasures} we prove lemmas concerning the tracing map and the refractor measure to be used in solving the problem. Section \ref{subsec:geodesicdisks} contains a few results about geodesic disks that are needed in the proof of the Lipschitz estimates.
The algorithm is explained in detail in Section \ref{sec:algorithm}, and the convergence in a finite number of steps in Section \ref{sec:lipschitzestimateimpliesstop}.
Section \ref{sec:lipschitzestimates} contains the Lipschitz estimate in Proposition \ref{prop:estimateofGiisimpler} needed to show the convergence of the algorithm in a finite number of steps.
Finally, in Section \ref{sec:numericalanalysis} we give a numerical implementation of our algorithm to construct various examples.


\section{Set up, definitions, and statement of results}\label{sec:preliminaries}

Suppose $\Gamma$ is a surface in $\R^3$ that separates two homogeneous, isotropic and dielectric media I and II having refractive indices $n_1$ and $n_2$, respectively. 
If a ray of light having direction $x \in S^2,$ the unit sphere in $\R^3,$ and traveling through the medium I strikes $\Gamma$ at the point P, then this ray is refracted in the direction $m \in S^2$ through the medium II according to the law of refraction (Snell's Law) 
\begin{equation}\label{vectorformofsnellslaw}
n_1(x \times \nu) =  n_2(m \times \nu),
\end{equation}
where $\nu$ is the unit normal to $\Gamma$ at $P$ pointing towards medium II.
%
If we set $\kappa = n_2/n_1$, then we can also write \eqref{vectorformofsnellslaw} as
\begin{equation}\label{vectorformofsnellslawwithlambda}
x - \kappa m = \lambda \nu
\end{equation}
where $\lambda \in \R$ is given by
$
\lambda = x \cdot \nu - m \cdot \nu
= x \cdot \nu-\kappa \sqrt{1-\kappa^{-2}(1-(x \cdot \nu)^{2})}$.
When medium I is optically denser than medium II, that is, $\kappa <1,$ 
the vector $m$ bends away from the normal, and total internal reflection might occur. 
That is, the ray with direction $m$ is transmitted to medium II if and only if $x\cdot m\geq \kappa$, or equivalently $x\cdot \nu\geq \sqrt{1-\kappa^2}$; see \cite[Section 2.1]{gutierrez-huang:farfieldrefractor}.
When $\kappa<1$, the surfaces having the uniform refracting property, are ellipsoids of revolution having a focus at the origin, see \cite[Section 2.2]{gutierrez-huang:farfieldrefractor}. That is, the surface written in polar coordinates $\rho(x)x$ with $x\in S^2$ and with
\begin{equation} \label{asurfacewithuniformrefraction}
\rho(x) = \dfrac{b}{1 - \kappa m \cdot x},
\end{equation}
$b>0$, is an ellipsoid of revolution with axis $m$, eccentricity $\kappa$, foci $0$ and $\dfrac{2\kappa b}{1 - \kappa^2}m$, and refracts all rays emanating from $0$ into the direction $m$ for $x\cdot m\geq \kappa$.
We then denote this semi-ellipsoid by
\begin{equation}\label{polarformofsemiellipsoid}
E(m,b) = \left\{\rho(x)x:\rho(x) = \dfrac{b}{1-\kappa m \cdot x},\  x \in S^2,\ m \cdot x \geq \kappa\right\}.
\end{equation}
{\it We assume throughout the paper that medium I is denser than medium II and therefore $\kappa = n_2/n_1 < 1.$} We also point out that similar analysis can be done for the case $\kappa = n_2/n_1 > 1$, changing ellipsoids for hyperboloids, see \cite[Section 2.2]{gutierrez-huang:farfieldrefractor}.

Suppose that $ \Omega$ and $ \Omega^*$ are two domains of the unit sphere $ S^2$ of $\R^3$\footnote{The physical problem considered is three dimensional; the mathematical extension to $n$ dimensions is straightforward.} with the property, to avoid total reflection \cite[Sect. 1.5.4]{book:born-wolf}, that
\begin{equation}\label{geometricconstraintfork<1}
\inf_{m \in \bar\Omega^*, x \in \bar \Omega} m \cdot x \geq \kappa,
\end{equation}
where $m \cdot x$ is the usual inner product of $m$ and $x$ in $\R^3$; and the boundary of $\Omega$ has surface measure zero.


\begin{definition}
A surface $\mathcal{R}$ in $\R^3$ parameterized by $\rho(x)x$ 
is a refractor from $\bar \Omega$ to $\bar \Omega^* $ if for any $x_o \in \bar{\Omega}$ there exists a semi-ellipsoid $E(m,b)$ with $m \in \bar \Omega^*$ such that $\rho(x_o)= \dfrac{b}{1-\kappa m \cdot x_o}$ and $\rho(x) \leq \dfrac{b}{1-\kappa m \cdot x}$ for all $x \in \bar{\Omega}.$ We call $E(m,b)$ a supporting semi-ellipsoid to $\mathcal R$ at $\rho(x_o)x_o$ or simply at $x_o.$
\end{definition}
From the definition, it is easy to see that refractors are Lipschitz continuous in $\bar \Omega$, i.e., $|\rho(x)-\rho(y)|\leq C_\kappa\,\left(\inf_{\Omega}\rho\right)\,|x-y|$ for $x,y,\in \bar \Omega$ with $C_\kappa$ a constant depending only on $\kappa$.
\begin{definition}
Given a refractor $\mathcal R = \{\rho(x)x: x \in \bar \Omega\},$ the refractor mapping of $\mathcal R$ is the multi-valued map defined for $x_{o} \in \bar \Omega$ by
$$
\n_{\r}(x_{o})=\{m \in \M : E(m,b)\  supports \ \r \ at \p(x_{o})x_{o}\ for \ some \ b>0\}.
$$
Given $m_{o} \in \bar \Omega^*$ the tracing mapping of $\r$ is defined by
\[
\t_{\r}(m_{o})=\{x \in \m: m_{o} \in \n_{\r}(x)\}.
\]
\end{definition}

Suppose that we are given a nonnegative function $g \in L^1(\bar\Omega)$.
We then recall the notion of refractor measure, see \cite[Section 3.1]{gutierrez-huang:farfieldrefractor}.
%
%
%
\begin{definition}\label{sol1}
The refractor measure associated with the refractor $\mathcal R$ and the function $g$ is the Borel measure given by 
\[
G_{\mathcal R}(\omega)= \int_{\t_{\mathcal R}(\omega)}g(x)\,dx
\]
for every Borel subset $\omega$ of $\M.$
\end{definition}
Given a Borel measure $\mu$ in $\Omega^*$ satisfying the energy conservation condition $\int_\Omega g(x)\,dx=\mu(\Omega^*)$, the {\it far field refractor problem} consists in finding a refractor $\r$ from $\Omega$ to $\Omega^*$ such that 
$G_\r=\mu$ in $\Omega^*$. Existence of refractors and uniqueness up to dilations is proved in \cite{gutierrez-huang:farfieldrefractor} using mass transport techniques. This is also proved in \cite{gutierrez-mawi:refractorwithlossofenergy} with a different method where a more general case that takes into account internal reflection is considered.

For the remaining part of the discussion fix $m_1, m_2, \ldots, m_N,$  $N \geq 2$, distinct points in $\M\subset S^2.$
Given $\b = (b_1, \ldots, b_N) \in \R_+^N$, i.e., with each $b_i > 0,$ we denote by $\r(\b)$ the refractor defined by a finite number of semi-ellipsoids and given by
\begin{equation}\label{refractordefinedbyfiniteellipsoids}
\r(\b) = \left\{\p(x)x: x \in \m, \p(x) = \min_{1 \leq i \leq N} \dfrac{b_i}{1-\kappa m_i \cdot x}\right\}.
\end{equation}
In this setting, we recall the following theorem from \cite[Remark 6.10]{gutierrez-mawi:refractorwithlossofenergy} for discrete targets.
\begin{theorem}\label{existence1}
Let $g \in L^1(\m)$ with $g > 0$ a.e., $f_1, \ldots, f_N$ are positive numbers, $m_1,\cdots ,m_N\in S^2$ are distinct points with $x\cdot m_j\geq \kappa$ for all $x\in \Omega$ and $1\leq j\leq N$.
Assume the energy conservation condition
\begin{equation}\label{energyconserved}
\int_{\m} g(x)\,dx = f_1+\cdots +f_N.
\end{equation}
Then there exists a refractor unique up to dilations\footnote{The assumption $g>0$ a.e. is only used to prove uniqueness up to dilations.}, having the form \eqref{refractordefinedbyfiniteellipsoids}, and solving $G_{\r(\b)}(m_i)  = f_i$ for all $i = 1, \ldots, N$.
\end{theorem}

{\it The main result of this paper is to describe an iterative scheme to construct this refractor with arbitrary precision.} That is, given $g\in L^\infty$ non negative, and $f_1,\cdots ,f_N;m_1,\cdots ,m_N,$ as in Theorem \ref{existence1}, and  $\epsilon > 0$ we find a vector $\b \in \R_+^N$, which depends on $\epsilon$, such that the refractor $\r(\b)$ of the form \ref{refractordefinedbyfiniteellipsoids} satisfies
\begin{equation}\label{conditionfornumericalsolution}
|G_{\r(\b)}(m_i) - f_i| \leq \epsilon,\qquad 1\leq i\leq N.
\end{equation}

\section{Preliminary results}


\subsection{Lemmas for the tracing map and refractor measures}\label{subset:lemmasfortracingmapandrefractormeasures}

\begin{lemma}\label{forlarge_b_energyis0}
Let $\b = (b_1, \ldots, b_N) \in \R^N$ with each $b_i > 0.$ Consider the family of refractors obtained from $\r(\b)= \{ \rho(x)x : x \in \Omega\},$ by changing only $b_i$ and fixing $b_j$ for all $j \neq i.$ Then:
\begin{itemize}
\item [i.] $G_{\r(\b)}(m_i) = 0$ for $b_i > (1+\kappa)\min_{j\neq i}b_j.$
\item [ii.]$G_{\r(\b)}(m_i) =\int_\Omega g(x)\,dx$ for  $0<b_i<\dfrac{\min_{j\neq i}b_j}{1+\kappa}.$ 
\end{itemize}
\end{lemma}
\begin{proof}
To prove (i) suppose $x \in \t_{\r(\b)}(m_i)$, and $x$ is not a singular point of $\r(\b)$. Then $E(m_i, b_i)$ is a supporting semi-ellipsoid to $\r(\b)$ at $\rho(x)x.$ 
So we have
\[
\dfrac{b_i}{1-\kappa m_i \cdot x} \leq  \dfrac{b_j}{1-\kappa m_j\cdot x}
\]
for all $j= 1, \ldots, N.$ Therefore
\begin{equation*}\label{bound on b_i}
b_i \leq \dfrac{1-\kappa m_i \cdot x}{1-\kappa m_j \cdot x} b_j 
\leq 
\dfrac{1-\kappa^2}{1-\kappa}b_j
=(1+\kappa)b_j,\,j=1,\cdots ,N.
\end{equation*}
Hence if $b_i > (1+\kappa)\min_{j\neq i}b_j$, then $\t_{\r(\b)}(m_i) \subset S,$ where $S$ is the singular set of $\mathcal R (\b).$
The first part of the lemma is then proved.
\noindent

Let us prove (ii). 
Let $b_0=\min_{j\neq i}b_j$, and take $0<b_i<b_0/(1+\kappa)$.
Then for any $x \in \Omega$ and for any $j \neq i$ we have
\begin{align*}
\dfrac{b_i}{1-\kappa m_i \cdot x} < \dfrac{b_0/(1+\kappa)}{1-\kappa m_i \cdot x}
\leq \dfrac{b_o}{1 - \kappa^2}
\leq \dfrac{b_j}{1 - \kappa^2} \leq \dfrac{b_j}{1-\kappa m_j \cdot x}.
\end{align*}
So for $0<b_i<b_0/(1+\kappa)$ we obtain
\[
\min_{1 \leq l \leq N}\dfrac{b_l}{1-\kappa m_l \cdot x} = \dfrac{b_i}{1-\kappa m_i \cdot x}
\]
and consequently $\t_{\r(\b)}(m_i) = \m$ completing the proof of part (ii) of the Lemma.
\end{proof}

\begin{remark}\label{rmk:Gconstantoutsideasetdefinedbylinearinequalities}\rm
For each fixed $1\leq i\leq N$, from Lemma \ref{forlarge_b_energyis0}, the function $G_{\r(\b)}(m_i)$ is constant on the set defined by linear inequalities 
\[
F_i:=\bigcup_{j\neq i}\left\{\b=(b_1,\cdots ,b_N): b_i\geq (1+\kappa)\,b_j\right\}\bigcup 
\bigcap_{j\neq i}\left\{\b=(b_1,\cdots ,b_N): b_i\leq \dfrac{1}{1+\kappa}\,b_j\right\}.
\]
If we set $G_i(\b)=G_{\r(\b)}(m_i)$, $1\leq i \leq N$, and consider the map $\b=(b_1,\cdots ,b_N)\mapsto \left(G_1(\b),\cdots ,G_N(\b)\right)$, the Jacobian of this map is zero on the set $\cup_{i=1}^NF_i$.
\end{remark}
\begin{lemma}\label{monotonicityoftrace}
Let $\b = (b_1, \ldots, b_N)$ and $\b^* = (b^*_1, \ldots, b_N^*)$ be in $\R^N_+.$ 
Suppose that for some $l$, $b^*_l \leq b_l$ and for all $i \neq l,$ $b_i^* = b_i$, where $1 \leq l, i\leq N.$ Then
\begin{equation}\label{eq:inclusionwithell}
\t_{\r(\b)}(m_l) \subseteq \t_{\r(\b^*)}(m_l)
\end{equation}
and 
\begin{equation}\label{eq:inclusionwithinotell}
\t_{\r(\b^*)}(m_i) \subseteq \t_{\r(\b)}(m_i) \  \textrm{for} \ i \neq l,
\end{equation}
where the inclusions are up to a set of measure zero.
Consequently
\[
G_{\r(\b)}(m_l) \leq G_{\r(\b^*)}(m_l) \ \textrm{and} \ G_{\r(\b)}(m_i) \geq G_{\r(\b^*)}(m_i) \  \textrm{for} \ i \neq l.
\]
\end{lemma}
\begin{proof}
We use here that if $x_0\in \t_{\r(\b)}(m_l)$ and $x_0$ is not a singular point, then the ellipsoid $E(m_l,b_l)$ supports $\r(\b)$ at $x_0$, this holds for any refractor $\r(\b)$ and any $1\leq l\leq N$; see \cite[Lemma 5.1]{gutierrez-mawi:refractorwithlossofenergy}.\footnote{The restriction that $x_0$ is not a singular point cannot be disposed of. For example, consider a refractor $\r$ that is the min of only two semi-ellipsoids $E(m_2,b_2)$ and $E(m_3,b_3)$. Take a singular point $x_0$ of this refractor and consider a supporting semi-ellipsoid $E(m_1,b)$ at $x_0$ having another direction $m_1$.  Take now $E(m_1,b_1)$ with $b_1>b$.
The refractor can be defined with the three ellipsoids $E(m_i,b_i)$, $1\leq i\leq 3$, because the definition of refractor does not see $E(m_1,b_1)$, but $E(m_1,b)$ is supporting at $x_0$ and $b<b_1$ and $E(m_1,b_1)$ does not support $\r$ at $x_0$.}

We first prove \eqref{eq:inclusionwithinotell} when $x_0$ is not a singular point of $\r(\b^*)$. Since $b_l^*\leq b_l$, we obviously have $\rho^*(x)\leq \rho(x)$ for all $x\in \Omega$, where $\rho^*$ is the parametrization of $\r(\b^*)$ and $\rho$ is the parametrization of $\r(\b)$. Suppose $i\neq l$ and let 
$x_0\in \t_{\r(\b^*)}(m_i)$. Then, since $x_0$ is not a singular point of $\r(\b^*)$, the ellipsoid with polar radius $\dfrac{b_i}{1 - \kappa x\cdot m_i}$ supports $\r(\b^*)$ at $x_0$. We have $\rho(x)\leq \dfrac{b_i}{1 - \kappa x\cdot m_i}$. Therefore 
\[
\dfrac{b_i}{1 - \kappa x_0\cdot m_i}=\rho^*(x_0)\leq \rho(x_0)\leq \dfrac{b_i}{1 - \kappa x_0\cdot m_i},
\]
that is, $x_0\in \t_{\r(\b)}(m_i)$. 
\newline
We now prove \eqref{eq:inclusionwithell}. That is, if $x_0$ is neither a singular point of $\r(\b)$ nor a singular point of $\r(\b^*)$, and $x_0\in \t_{\r(\b)}(m_l)$, then $x_0\in \t_{\r(\b^*)}(m_l)$. We may assume $b_l^*<b_l$.
We have that $E(m_l,b_l)$ supports $\r(\b)$ at $x_0$. We claim that the ellipsoid with polar radius $\dfrac{b_l^*}{1-\kappa\,x\cdot m_l}$ supports $\r(\b^*)$ at $x_0$. 
Suppose this is not true. Since by definition $\rho^*(x)\leq \dfrac{b_l^*}{1-\kappa\,x\cdot m_l}$, we would have $\rho^*(x_0)< \dfrac{b_l^*}{1-\kappa\,x_0\cdot m_l}$.
So $\rho^*(x_0)=\dfrac{b_j}{1-\kappa\,x_0\cdot m_j}$ for some $j\neq l$, and therefore $\dfrac{b_j}{1-\kappa\,x\cdot m_j}$  supports $\r(\b^*)$ at $x_0$.
Since $x_0$ is not a singular point of $\r(\b^*)$, then by the inclusion previously proved we get that $x_0\in \t_{\r(\b)}(m_j)$. Since $j\neq l$ we obtain that $x_0$ is a singular point of $\r(\b)$, a contradiction.
%
\end{proof}

\begin{remark}\rm
We show that if $\Omega$ is connected, $0<|\t_{\r(\b)}(m_l)|<|\Omega|$, and $b_\ell^*<b_\ell$, then 
$|\t_{\r(\b^*)}(m_l)\setminus \t_{\r(\b)}(m_l)|>0$. 
Therefore, if $g>0$ a.e., then this implies that if $0<G_{\r(\b)}(m_l)<\int_\Omega g(x)\,dx$  we obtain $G_{\r(\b)}(m_l)< G_{\r(\b^*)}(m_l)$ when $b_\ell^*<b_\ell$.

In fact, the proof follows the argument in \cite[Lemma 4.12]{gutierrez:cimelectures}.
Since $|\t_{\r(\b)}(m_l)|>0$, by \cite[Lemma 4.11]{gutierrez:cimelectures} if $x_0\in  \t_{\r(\b)}(m_l)$, then the semi-ellipsoid $E(m_\ell,b_\ell)$ supports $\r(\b)$ at $x_0$.
Hence $\dfrac{b_j}{1-\kappa\,m_j\cdot x_0}\geq \dfrac{b_\ell}{1-\kappa\,m_\ell\cdot x_0}$ for all $j$. Since $b_\ell^*<b_\ell$, we then get
\[
\dfrac{b_j}{1-\kappa\,m_j\cdot x_0}> \dfrac{b_\ell^*}{1-\kappa\,m_\ell \cdot x_0}\qquad \forall j.
\]
By continuity there is a neighborhood $V_{x_0}$ such that 
\[
\dfrac{b_j}{1-\kappa\,m_j\cdot x}> \dfrac{b_\ell^*}{1-\kappa\,m_\ell \cdot x}\qquad \forall j,\quad \forall x\in V_{x_0}.
\]
Thus $V_{x_0}\subset \t_{\r(\b^*)}(m_l)$. Therefore,
we have the inclusion 
\[
\t:=\t_{\r(\b)}(m_l)\subset \text{\it interior}\left(\t_{\r(\b^*)}(m_l) \right):=\mathcal O.
\]
On the other hand, from the proof of \cite[Lemma 3.12]{gutierrez:cimelectures}, the set $\t_{\r(\b)}(m_l)$ is compact. Therefore the set $\mathcal O\setminus \t_{\r(\b)}(m_l)$ is an open set. 
Now since $G_{\r(\b)}(m_l)<\int_\Omega g(x)\,dx$, by the continuity of the refractor measure as a function of $b_\ell$, Lemma \ref{lm:continuity}(ii) below, we have that $G_{\r(\b^*)}(m_l)<\int_\Omega g(x)\,dx$ 
for $b_\ell^*$ sufficiently close to $b_\ell$. 
Since $G_{\r(\b^*)}(m_l)$ increases when $b_\ell^*$ decreases, it is enough to prove the desired inequality when  $b_\ell^*$ is sufficiently close to $b_\ell$.
This implies that $\mathcal O\neq \bar \Omega$.
So we have the configuration $\t$ closed, $\t\subset \mathcal O\subsetneqq \bar \Omega$.
If the set $\mathcal O\setminus \t\neq \emptyset$, then since $\mathcal O\setminus \t$ is open, we have $|\mathcal O\setminus \t|>0$. Since $\mathcal O\setminus \t\subset \t_{\r(\b^*)}(m_l)\setminus \t_{\r(\b)}(m_l)$, we obtain the desired result. 
It then remains to show that $\mathcal O\setminus \t\neq \emptyset$.
Suppose by contradiction that $\mathcal O\setminus \t= \emptyset$.
That is, $\mathcal O\cap (\bar \Omega\setminus \t)=\emptyset$.
We shall prove this implies that 
\begin{equation}\label{eq:omegadisconnected}
\bar \Omega=\mathcal O\cup(\bar \Omega\cap \t^c).
\end{equation} 
Since both sets in this union are open ($\t$ is closed) relative to $\bar \Omega$, and $\mathcal O\neq \bar \Omega$, we obtain that $\bar \Omega$ is disconnected, contradicting the assumption that $\bar \Omega$ is connected.
So let us prove \eqref{eq:omegadisconnected}. Write
\begin{align*}
\mathcal O\cup(\bar \Omega\cap \t^c)
&=
(\mathcal O \cup \bar \Omega)\cap (\mathcal O \cup \t^c)
=\bar \Omega\cap (\mathcal O \cup \t^c)\\
&\supset \bar \Omega \cap (\mathcal O \cup \mathcal O^c)=\bar \Omega\qquad \text{since $\t\subset \mathcal O$}.
\end{align*}
This completes the remark.
\end{remark}

By \cite[Lemma 3.6]{gutierrez-mawi:refractorwithlossofenergy}, we also have the following:
\begin{lemma}\label{convergence of refractors}
Let $\mathcal R_{j} = \{\p_{j}(x)x: x \in \m\}$, $j \geq 1$ be refractors from $\m$ to $\M.$ Suppose that $0 < a_1 \leq \p_{j} \leq a_2$ and $\p_{j} \to \p$ pointwise on $\m.$ 
Then:
\begin{itemize}
\item[i.] $\mathcal R := \{\p(x)x: x \in \m\}$ is a refractor from $\m$ to $\M.$
\item[ii.]The measures $G_{\r_j}$ converge weakly to the measure $G_{\r}.$
\end{itemize}
\end{lemma}
\begin{lemma}\label{lm:continuity}
If in Lemma \ref{convergence of refractors}, $\r_j$ and $\r$ are defined by finite number of semi-ellipsoids as:
\[
\r_j=\r(\b_j) = \left\{\p(x)x: x \in \m, \p(x) = \min_{1 \leq i \leq N} \dfrac{b^j_i}{1-\kappa m_i \cdot x}\right\}.
\]
and
\[
\r=\r(\b) = \left\{\p(x)x: x \in \m, \p(x) = \min_{1 \leq i \leq N} \dfrac{b_i}{1-\kappa m_i \cdot x}\right\}.
\]
then
\begin{itemize}
\item[i.] $G_{\r_j} = \sum G_{\r_j}(m_i) \delta_{m_i}, G_{\r} = \sum G_{\r}(m_i) \delta_{m_i}\ $
\item[ii.] $G_{\r_j}(m_i) \to G_{\r}(m_i)$ for all $1\leq i\leq N$, when $\b_j\to \b$.
\end{itemize}
\end{lemma}
For a proof of this lemma see \cite[Lemma 4.7]{gutierrez:cimelectures}.

\subsection{Geodesic disks}\label{sec:analysisofgeodesicdisks}
\label{subsec:geodesicdisks}

Recall that if $\alpha,\beta\in S^2,$ the geodesic distance between them is given by $\cos^{-1} (\alpha \cdot \beta).$ We define a geodesic disk with center $\alpha$ and radius $r$ to be the set of points $x$ on $S^2$ for which $x \cdot \alpha \geq \cos r.$

\begin{lemma}\label{lm:descriptionofgeodesicdisk}
Let $b_1,b_2>0$ and $m_1,m_2\in S^2$ be such that $m_1\neq m_2$.
Consider the set 
\[
V_{12}=\left\{x\in S^2:\dfrac{b_1}{1-\kappa\,x\cdot m_1}\leq \dfrac{b_2}{1-\kappa\,x\cdot m_2} \right\}.
\]
This set is non empty if and only if $\dfrac{b_1-b_2}{\kappa\,|b_1m_2-b_2m_1|}\leq 1$, and 
$V_{12}$ is the geodesic disk with center at
\[
A_{12}=\dfrac{b_1m_2-b_2m_1}{|b_1m_2-b_2m_1|}
\]
and radius
\[
\tau_{12} =\cos^{-1} \dfrac{b_1 - b_2}{\kappa|b_1 m_2 - b_2 m_1|},
\]
that is,
\[
V_{12}=\left\{x\in S^2:x\cdot A_{12}\geq \cos \tau_{12}\right\}.
\]
In addition, if $\dfrac{b_1-b_2}{\kappa\,|b_1m_2-b_2m_1|}\leq -1$, then $V_{12}=S^2$.
If $\r(\b)$ is the refractor in $\Omega$ with polar radius $\rho(x)=\min\left\{\dfrac{b_1}{1-\kappa \,x\cdot m_1},\dfrac{b_2}{1-\kappa \,x\cdot m_2} \right\}$, then 
$\t_{\r(\b)}(m_1)\subset V_{12}$.
\end{lemma}
\begin{proof}
If $\t_{\r(\b)}(m_1)=\emptyset$, there is nothing to prove.
Otherwise, let $x \in \t_{\r(\b)}(m_1)$. Then $\dfrac{b_1}{1 - \kappa m_1 \cdot x } \leq\dfrac{b_2}{1 - \kappa m_2 \cdot x}$. So $b_1 - b_2 \leq x \cdot \kappa (b_1m_2 - b_2m_1)$, and we obtain 
\[
x \cdot \dfrac {b_1m_2 - b_2m_1}{|b_1 m_2 - b_2 m_1|} \geq \frac{1}{\kappa}\dfrac{b_1-b_2}{|b_1 m_2 - b_2 m_1|},
\]
in particular, we must have $\dfrac{b_1-b_2}{\kappa\,|b_1 m_2 - b_2 m_1|}\leq 1$.
Thus $\t_{\r(\b)}(m_1)$ is contained in the geodesic disc with center at $\dfrac {b_1m_2 - b_2m_1}{|b_1 m_2 - b_2 m_1|}$ and geodesic radius $\cos^{-1} \left(\dfrac{1}{\kappa}\dfrac{b_1-b_2}{|b_1 m_2 - b_2 m_1|}\right).$
\end{proof}

\begin{remark}\label{rmk:formulatauintermsofVj}\rm
If $\r(\b)$ is a refractor of the form \eqref{refractordefinedbyfiniteellipsoids}, then 
\begin{equation}\label{eq:identityoutsidesingularpoints}
\t_{\r(\b)}(m_i)=\Omega\cap \cap_{j=1}^NV_{ij}, \,\text{  except possibly on the singular set of $\r(\b)$,}
\end{equation}
where $V_{ij}=\left\{x\in S^2:\dfrac{b_i}{1-\kappa\,x\cdot m_i}\leq \dfrac{b_j}{1-\kappa\,x\cdot m_j} \right\}$.
In fact, if $x_0\in \t_{\r(\b)}(m_i)$ and $x_0$ is not singular, then by \cite[Lemma 5.1]{gutierrez-mawi:refractorwithlossofenergy} the semi-ellipsoid $E(m_i,b_i)$ supports $\r(\b)$ at $x_0$ implying 
$\dfrac{b_i}{1-\kappa\,x_0\cdot m_i}\leq \dfrac{b_j}{1-\kappa\,x_0\cdot m_j}$ for all $j$.
Vice versa, if $x_0\in \Omega\cap \cap_{j=1}^NV_{ij}$, then the polar radius $\rho$ satisfies 
$\rho(x_0)=\dfrac{b_i}{1-\kappa\,x_0\cdot m_i}$, and so $\dfrac{b_i}{1-\kappa\,x\cdot m_i}$ supports $\rho$ at $x_0$.

The following example shows that in \eqref{eq:identityoutsidesingularpoints} it is necessary to remove the singular points. In fact, 
take two ellipsoids $E_1$ and $E_2$ with polar radii 
$\dfrac{b_1}{1-\kappa \, x\cdot m_1}$ and $\dfrac{b_2}{1-\kappa \,x\cdot m_2}$ respectively,
with $m_1\neq m_2$ and take the corresponding refractor minimum of the two ellipsoids and let $\rho(x)$ be the polar radius.
Suppose the refractor $\rho$ has a singular point $x_0$.
At $x_0$ take a supporting semi-ellipsoid to $\rho$ having axis $m_3$, with $m_3$ different from $m_1$ and $m_2$. Let the polar radius of this ellipsoid be $\dfrac{b_3}{1-\kappa x\cdot m_3}$.
Now take an ellipsoid $E$ of the form $\dfrac{b*}{1-\kappa x\cdot m_3}$ with $b*$ much larger than $b_3$ so that the ellipsoids $E_1$ and $E_2$ are contained in the interior of the solid $E$, that is, $\dfrac{b_i}{1-\kappa\, x\cdot m_i}$, $i=1,2$, are both strictly smaller than $\dfrac{b*}{1-\kappa x\cdot m_3}$.
Then refractor $\min\left\{\dfrac{b_1}{1-\kappa \, x\cdot m_1}, \dfrac{b_2}{1-\kappa \,x\cdot m_2},\dfrac{b*}{1-\kappa x\cdot m_3}\right\}$
is the same as the refractor $\rho(x)$.
We have that $x_0\in \t_\rho(m_3)$.
On the other hand, the sets 
$V_{31}=\left\{\dfrac{b*}{1-\kappa x\cdot m_3}\leq \dfrac{b_1}{1-\kappa \, x\cdot m_1}\right\}=\emptyset$ and 
$V_{32}=\left\{\dfrac{b*}{1-\kappa x\cdot m_3}\leq \dfrac{b_2}{1-\kappa \,x\cdot m_2}\right\}=\emptyset$.
\end{remark}
\vskip 0.2in

\section{The algorithm}\label{sec:algorithm}
We assume the energy conservation condition \eqref{energyconserved}.
\subsection{The set $W$ of admissible vectors}\label{sec:admissiblevectors}
Let $N\geq 2$, $f_o=\min_{1\leq i\leq N}f_i$, and $0<\delta<f_o/N$.
Consider the set of admissible vectors 
\[
W = \{\b = (1,b_2,\ldots, b_N): b_i > 0 \, \textrm {and} \, G_{\r(\b)}(m_i) \leq f_i + \delta \, \,  \textrm{for} \,  i = 2, \ldots, N\}.
\]
This set is non empty and their coordinates are bounded away from zero. This is the contents of the following lemma.
\begin{lemma}\label{lm:Wnonemptyandlowerbound}
Suppose $f_o=\min_{1\leq i\leq N}f_i$ and $0<\delta <f_o/N$.
We have that
\begin{enumerate}
\item[(1)] if $b_i>1+\kappa$ for $2\leq i\leq N$, then $(1,b_2,b_3,\cdots ,b_N)\in W$;
\item[(2)] 
if $\b=(1,b_2,\cdots ,b_N)\in W$, then
\begin{equation}\label{eq:lowerboundofbi}
\text{$b_i\geq \dfrac{1}{1+\kappa}$ for $2\leq i\leq N$.}
\end{equation}
\end{enumerate}
\end{lemma}
\begin{proof}
We prove (1).
Let $\b=(1,b_2,\cdots ,b_N)$ with $b_i>0$. Fix $j \geq 2$ and let $x\in \t_{\r(\b)}(m_j)$ be a non singular point. Then from \cite[Lemma 5.1]{gutierrez-mawi:refractorwithlossofenergy}, the semi-ellipsoid $E(m_j,b_j)$ supports $\r(\b)$ at $x$.
Since $x\cdot m_j\geq \kappa$, we have 
\[
\dfrac{b_j}{1-\kappa^2}\leq \dfrac{b_j}{1-\kappa \, x\cdot m_j}=\rho(x)\leq \dfrac{1}{1-\kappa \, x\cdot m_1}\leq \dfrac{1}{1-\kappa},
\]
and so $b_j\leq 1+\kappa$.
Therefore, if $b_j>1+\kappa$ with $j\neq 2$, and $x\in \t_{\r(\b)}(m_j)$, then $x$ is a singular point and therefore  
$\mathcal T_{\mathcal R}(m_j)$ has measure zero, and so
$G_{\mathcal R(b)}(m_j)=0<f_j+\delta$. 

To show (2), we first prove that $G_{\r(\b)}(m_1)>0$ for each $\b\in W$. 
In fact, from (\ref{energyconserved}) and the definition of $W$ we have
\[
G_{\r(\b)}(m_1) = f_1 +\sum_{i = 2}^N (f_i - G_{\r(\b)}(m_i)) > f_1 - (N-1)\delta > f_1 - N\delta>0
\]
from the choice of $\delta$.
Since $g\geq 0$, the set $\mathcal T_{\r(\b)}(m_1)$ has positive measure.
This implies that for each $\b\in W$, $\mathcal T_{\mathcal R(\b)}(m_1)\cap \left(\cup_{i=2}^N\mathcal T_{\mathcal R(\b)}(m_i)\right)^c\neq \emptyset$.
Otherwise, $\mathcal T_{\mathcal R(\b)}(m_1)\subset \cup_{i=2}^N\mathcal T_{\mathcal R(\b)}(m_i)$ which means that each point in $\mathcal T_{\mathcal R(\b)}(m_1)$ is singular, and therefore $|\mathcal T_{\mathcal R(\b)}(m_1)|=0$; a contradiction.
From this we conclude \eqref{eq:lowerboundofbi} because, if $\b\in W$, then we can pick $x_0\in \mathcal T_{\mathcal R(\b)}(m_1)\cap \left(\cup_{i=2}^N\mathcal T_{\mathcal R(\b)}(m_i)\right)^c$
and we have
\[
\rho(x_0)=\dfrac{1}{1-\kappa\,x_0\cdot m_1}\leq \dfrac{b_i}{1-\kappa\,x_0\cdot m_i},\qquad i=2,\cdots ,N\]
so
\[
b_i\geq \dfrac{1-\kappa\,x_0\cdot m_i}{1-\kappa\,x_0\cdot m_1}\geq \dfrac{1-\kappa\,x_0\cdot m_i}{1-\kappa^2}\geq 
\dfrac{1-\kappa}{1-\kappa^2}=\dfrac{1}{1+\kappa}.
\]
\end{proof}
\subsection{Detailed description of the algorithm}\label{subsect:detaileddescriptionofthealgorithm}
From Lemma \ref{lm:Wnonemptyandlowerbound} (2),
we can pick $\b^1=(1,b_2,\cdots ,b_N)\in W$. We will construct $N-1$ intermediate consecutive vectors $\b^2,\cdots ,\b^N$ associated with $\b^1$ in the following way.

{\bf Step 1.}
We first test if $\b^1$ satisfies the inequality:
\begin{equation}\label{eq:testinequality}
f_2-\delta \leq G_{\r(\b^1)}(m_2)\leq f_2+\delta.
\end{equation}
If $\b^1$ satisfies this inequality, then we set $\b^2=\b^1$ and we proceed to Step 2 below.
Notice that the inequality on the right hand side of \eqref{eq:testinequality} holds since $\b^1\in W$.
If $\b^1$ does not satisfy \eqref{eq:testinequality}, then
\begin{equation}\label{eq:Glessthanf2-delta}
G_{\r(\b^1)}(m_2)<f_2-\delta. 
\end{equation}
We shall pick $b_2^*\in (0,b_2)$, and leave all other components fixed, so that the new vector 
$\b^2=(1,b_2^*,b_3,\cdots ,b_N)$ belongs to $W$, and satisfies
\begin{equation}\label{eq:Glessthanf2-deltastronger}
f_2\leq G_{\r(\b^2)}(m_2)\leq f_2+\delta.
\end{equation}
In fact, this is possible because applying Lemma \ref{monotonicityoftrace} with $\ell=2$ we get that 
$G_{\r(\b^2)}(m_j)\leq G_{\r(\b^1)}(m_j)$ for $j\neq 2$ and $b_2^*\in (0,b_2]$ from \eqref{eq:inclusionwithinotell}; and applying Lemma \ref{forlarge_b_energyis0} (ii.) we get that $G_{\r(\b^2)}(m_2)\to \int_\Omega g(x)\,dx=f_1+\cdots +f_N$ as $b_2^*\to 0$, from the energy conservation assumption. 
Since the $f_i$'s are all positive, $f_1+\cdots +f_N>f_2$, and from the choice of $\delta$ we have $f_1+\cdots +f_N>f_2+\delta$. 
As a function of $b_2^*$, the function $G_{\r(\b^2)}(m_2)$ is non-increasing on $(0,b_2)$ from \eqref{eq:inclusionwithell}, tends to $f_1+\cdots +f_N$ as $b_2^*\to 0$, and from \eqref{eq:Glessthanf2-delta} is strictly less than $f_2-\delta$ at $b_2^*=b_2$.
Therefore by continuity of $G_{\r(\b^2)}(m_2)$, Lemma \ref{lm:continuity} ii, we can pick a value $b_2^*\in (0,b_2)$ such that \eqref{eq:Glessthanf2-deltastronger} holds\footnote{Notice that for any $a\in [f_2-\delta,f_1+\cdots ,f_N]$, we can pick $b_2^*\in (0,b_2)$ such that $G_{\r(\b^2)}(m_2)=a$.}.
Therefore, if the vector $\b^1$ does not satisfy \eqref{eq:testinequality}, we have then constructed a vector $\b^2\in W$ that satisfies \eqref{eq:Glessthanf2-deltastronger} which is stronger than \eqref{eq:testinequality}.
\newline
{\bf Step 2.}
Next we proceed to test the inequality 
\begin{equation}\label{eq:testinequality2}
f_3-\delta\leq G_{\r(\b^2)}(m_3)\leq f_3+\delta,
\end{equation}
with $\b^2$ the vector constructed in Step 1.
If $\b^2$ satisfies \eqref{eq:testinequality2}, we set $\b^3=\b^2$ and we proceed to the next step.
If $\b^2$ does not satisfy \eqref{eq:testinequality2}, then 
\[
G_{\r(\b^2)}(m_3)< f_3-\delta
\]
and we proceed as before, now to decrease the value of $b_3$, the third component of the vector $\b^2$, and construct a vector $\b^3\in W$ such that 
\[
f_3\leq G_{\r(\b^3)}(m_3)\leq f_3+\delta,
\]
and in particular,
\eqref{eq:testinequality2} holds for $\b^3$.
Notice that we do not know if the newly constructed vector $\b^3$ satisfies \eqref{eq:testinequality}.
\newline
{\bf Step 3.}
Next we proceed to test the inequality 
\begin{equation}\label{eq:testinequality3}
f_4-\delta\leq G_{\r(\b^3)}(m_4)\leq f_4+\delta,
\end{equation}
with $\b^3$ the vector from Step 2. If this is true, then we set $\b^4=\b^3$ and proceed to the next step. Otherwise, we must have 
\[
G_{\r(\b^3)}(m_4)<f_4-\delta
\]
 and we continue in the same way as before now decreasing the fourth component $b_4$ of $\b^3$ obtaining a new vector $\b^4$ satisfying
\[
f_4\leq G_{\r(\b^4)}(m_4)\leq f_4+\delta,
\]
in particular, \eqref{eq:testinequality3}.
\newline
{\bf Step $N-1$.}
We proceed to test the inequality 
\begin{equation}\label{eq:testinequalityN-1}
f_N-\delta \leq G_{\r(\b^{N-1})}(m_N)\leq f_N+\delta,
\end{equation}
where $\b^{N-1}$ is the vector from Step $N-2$. If this holds we set $\b^N=\b^{N-1}$. Otherwise,  we have
\[
G_{\r(\b^{N-1})}(m_N)<f_N-\delta,
\]
and proceeding as before, by decreasing the $N$th-component of $\b^{N-1}$, we obtain a vector $\b^N\in W$
\[
f_N\leq G_{\r(\b^N)}(m_N)\leq f_N+\delta.
\]
\newline
In this way, starting from a fixed vector $\b^1\in W$, we have constructed intermediate vectors $\b^2,\cdots ,\b^N$ all belonging to $W$ and satisfying the above inequalities.
Notice that by construction, the $\ell$-th components of $\b^{j-1}$ and $\b^j$ are all equal for $\ell\neq j$.
If for some $2\leq j\leq N$, $\b^{j-1}\neq \b^j$, then the $j$-th component of $\b^j$ is strictly less than the $j$-th component of $\b^{j-1}$. And so if we needed to decrease the $j$-th component of $\b^{j-1}$ to construct $\b^j$ is because 
\[
G_{\r(\b^{j-1})}(m_j)<f_j-\delta,
\]
and then by construction $\b^j$ satisfies  
\[
f_j\leq G_{\r(\b^j)}(m_j)\leq f_j+\delta.
\]
We therefore obtain from the last two inequalities the following important inequality
\begin{equation}\label{eq:inequalityfordifferenceoftwobjs}
\delta< G_{\r(\b^j)}(m_j)-G_{\r(\b^{j-1})}(m_j),\quad \text{for intermediate vectors $\b^j\neq \b^{j-1}$.}
\end{equation}

We now repeat the construction above starting with the last vector $\b_N$.
In fact, we start from a vector $\b^{1,1}\in W$ and constructed $N-1$ intermediate vectors $\b^{1,2},\cdots ,\b^{1,N}$ using the procedure described.
So we obtain in the first step the finite sequence of vectors 
\[
\b^{1,1},\b^{1,2},\cdots , \b^{1,N}.
\]
In the second step we repeat the construction now starting with the vector $\b^{1,N}$ and we get the finite sequence of vectors 
\[
\b^{2,1},\b^{2,2},\cdots , \b^{2,N}
\] 
with $\b^{2,1}=\b^{1,N}$.
For the third step we repeat the process now starting with the last intermediate vector $\b^{2,N}$ obtained in the previous step, obtaining the finite sequence of vectors 
\[
\b^{3,1},\b^{3,2},\cdots , \b^{3,N}
\]
with $\b^{3,1}=\b^{2,N}$.
Continuing in this way we obtain a sequence of vectors, in principle infinite,
\begin{equation}\label{eq:sequenceofvectorsintheiteration}
\b^{1,1},\cdots , \b^{1,N};\b^{2,1},\cdots ,\b^{2,N}; \b^{3,1},\cdots ,\b^{3,N};\cdots ;\b^{n,1},\cdots ,\b^{n,N};\b^{n+1,1},\cdots ,\b^{n+1,N};\cdots 
\end{equation}
with $\b^{2,1}=\b^{1,N},\b^{3,1}=\b^{2,N},\cdots ,\b^{n+1,1}=\b^{n,N},\cdots$.  
If for some $n$, the vectors in the $n$th-stage are equal, i.e., 
$\b^{n,1}=\b^{n,2}=\cdots =\b^{n,N}:=\b^n$, then 
from the construction
\[
|G_{\r(\b^n)}(m_j)- f_j|\leq \delta,\quad \text{for $2\leq j\leq N$.} 
\]
Furthermore, by conservation of energy,
$
\sum_{i=1}^N G_{\r(\b^n)}(m_i) = \sum_{i=1}^N f_i,
$
so we obtain
\begin{align*}
|f_1-G_{\r(\b^n)}(m_1)|
&= \left|\sum_{j=2}^NG_{\r(\b^n)}(m_j) - f_j\right| 
\leq \sum_{j=2}^N |G_{\r(\b^n)}(m_j) - f_j| \leq N\,\delta.
\end{align*}
If we now choose $\delta = \epsilon/N$, then the refractor $\r(\b^n)$ will satisfy \eqref{conditionfornumericalsolution}, and the problem is solved.
\newline
Therefore, if we show that for some $n$ the intermediate vectors $\b^{1,n},\b^{2,n},\cdots,\b^{n,N}$ are all equal, we are done.

\subsection{A Lipschitz estimate implies that the process stops}\label{sec:lipschitzestimateimpliesstop}
We shall prove that the estimate \eqref{eq:estimateofGiisimpler}
implies that there is an $n$ such that the vectors in the group $\b^{n,1},\b^{n,2},\cdots,\b^{n,N}$ are all equal, and we also show an upper bound for the number of iterations. 

Suppose we originate the iteration at $\b^0=(1,b_2^0,\cdots ,b_N^0)  \in W.$ 
Since by construction the coordinates of the vectors in the sequence \eqref{eq:sequenceofvectorsintheiteration} are decreased or kept constant, the $j$th coordinate of any vector in the sequence is less than or equal to $b_j^0$, $1\leq j\leq N$.
In addition, from \eqref{eq:lowerboundofbi}, points in $W$ have all their coordinates  bounded  below by $1/(1+\kappa)$. Therefore all terms in the sequence \eqref{eq:sequenceofvectorsintheiteration} are contained in the compact box $K=\{1\}\times \prod_{j=2}^N [1/(1+\kappa),b_j^0]$.
We want to show that there is $n_0$ such that the intermediate vectors $\b^{n_0,1},\b^{n_0,2},\cdots ,\b^{n_0,N}$ are all equal.
Otherwise, for each $n$ the intermediate vectors $\b^{n,1},\b^{n,2},\cdots ,\b^{n,N}$ are not all equal.
This implies that for each $n$ there are two consecutive intermediate vectors $(1,b_2,b_3,\cdots,b_N)$ and 
$(1,\bar b_2,\bar b_3,\cdots,\bar b_N)$, that are different. By construction of intermediate vectors, they can only differ in one coordinate, say that $b_j>\bar b_j$. Notice that $j$ depends on $n$, but there is $j$ and a subsequence $n_\ell$ such that there are two consecutive  intermediate vectors $(1,b_2^{n_\ell},b_3^{n_\ell},\cdots,b_N^{n_\ell})$ and 
$(1,\bar b_2^{n_\ell},\bar b_3^{n_\ell},\cdots,\bar b_N^{n_\ell})$ in each group 
$\b^{n_\ell,1},\cdots ,\b^{n_\ell,N}$ such that their $j$-th coordinates satisfy 
$b_j^{n_\ell}>\bar b_j^{n_\ell}$, and all other coordinates are equal. Also notice that since the coordinates are chosen in a decreasing form we have $b_j^{n_\ell}>\bar b_j^{n_\ell}\geq b_j^{n_{\ell+1}}>\bar b_j^{n_{\ell+1}}$ for $\ell=1,\cdots $.
From \eqref{eq:inequalityfordifferenceoftwobjs} we then get
\begin{equation}\label{eq:Lipschitzbounddelta}
\delta
<G_j\left(1,\bar b_2^{n_\ell},\bar b_3^{n_\ell},\cdots,\bar b_N^{n_\ell}\right)-G_j\left(1,b_2^{n_\ell},b_3^{n_\ell},\cdots,b_N^{n_\ell}\right)=(*)
\end{equation}
for each $\ell\geq 1$. 
We write
\[
(1,\bar b_2^{n_\ell},\bar b_3^{n_\ell},\cdots,\bar b_j^{n_\ell},\cdots \bar b_N^{n_\ell})
=
(1,\bar b_2^{n_\ell},\bar b_3^{n_\ell},\cdots,b_j^{n_\ell}+\bar b_j^{n_\ell}-b_j^{n_\ell},\cdots \bar b_N^{n_\ell}),
\]
and let $t:=\bar b_j^{n_\ell}-b_j^{n_\ell}<0$.
Since the vectors belong to $W$, we have $\bar b_j^{n_\ell}\geq 1/(1+\kappa)$. Then from \eqref{eq:estimateofGiisimpler}
we obtain
\begin{equation}\label{eq:boundswithL}
(*)\leq -\left(\bar b_j^{n_\ell}-b_j^{n_\ell}\right)\,
C_\kappa\,(\sup_\Omega g)\,  (N-1):=L\,(b_j^{n_\ell}-\bar b_j^{n_\ell}),\qquad \forall \ell.
\end{equation}
On the other hand,
\begin{equation}\label{eq:estimateofsumofdifferencesofcoordinates}
\sum_{\ell=1}^\infty ( b_j^{n_\ell}-\bar b_j^{n_\ell})\leq b_j^0-\dfrac{1}{1+\kappa},
\end{equation}
which contradicts \eqref{eq:Lipschitzbounddelta} and therefore the intermediate vectors $\b^{n_0,1},\b^{n_0,2},\cdots ,\b^{n_0,N}$ are all equal for some $n_0$.

Let us now estimate the number of iterations used.
Consider the sequence of vectors \eqref{eq:sequenceofvectorsintheiteration} constructed and list them 
as a sequence denoted by $v_i$, $i=1,2,3,\cdots $ and maintaining the given order.
By construction the $j$-th coordinate of the vector $v_i$ is greater than or equal than the $j$-th coordinate of the vector $v_{i+1}$, $1\leq j\leq N$.
Given $1\leq j\leq N$, if we let 
$
c_j(v_i)=\text{$j$-th coordinate of the vector $v_i$},
$
then $c_j(v_i)\geq c_j(v_{i+1})$; and any two consecutive vectors $v_i$ and $v_{i+1}$ can differ in only one
coordinate.
Let $\mathcal C_j=\{i:c_j(v_i)> c_j(v_{i+1})\}$; (notice that $\mathcal C_1=\emptyset$). 
If $i\in \mathcal C_j$, then from \eqref{eq:Lipschitzbounddelta} and \eqref{eq:boundswithL}
\[
c_j(v_i)- c_j(v_{i+1})\geq \dfrac{\delta}{L},
\]
and so adding over $i$ we get from \eqref{eq:estimateofsumofdifferencesofcoordinates} 
\[
\#( \mathcal C_j)\leq \dfrac{L}{\delta}\left( b_j^0-\dfrac{1}{1+\kappa}\right).
\]  
Now the set $\{i:v_i\neq v_{i+1}\}\subset \cup_{j=2}^N \mathcal C_j$, and therefore the sequence \eqref{eq:sequenceofvectorsintheiteration} is constant for all $n\geq n_0$ with 
\begin{equation}\label{eq:numberoftotaliterations}
n_0\leq N\,\dfrac{(1+\kappa) \,C_\kappa\,(\sup_\Omega g)\,  (N-1)}{\delta}\,\max_{2\leq j\leq N}\left( b_j^0-\dfrac{1}{1+\kappa}\right).
\end{equation}

\subsection{Limit as $n\to \infty$ of the sequence \eqref{eq:sequenceofvectorsintheiteration}}
We will show here that the procedure described always converges in an infinite number of steps, 
assuming only that $g\in L^1(\Omega)$ with $g$ not necessarily bounded.
This can be clearly seen by listing the vectors constructed in the following way:
\[
\text{group 1 }\left\{
\begin{matrix}
\b^{1,1} & \to & 1 & b_2^{1,1} & b_3^{1,1} & b_4^{1,1} & \cdots & b_N^{1,1}\\
& &  \verteq & \downineq & \verteq & \verteq & \cdots & \verteq\\
\b^{1,2} & \to & 1 & b_2^{1,2} & b_3^{1,2} & b_4^{1,2} & \cdots & b_N^{1,2}\\
& &  \verteq & \verteq & \downineq & \verteq & \cdots & \verteq\\
\b^{1,3} & \to & 1 & b_2^{1,3} & b_3^{1,3} & b_4^{1,3} & \cdots &  b_N^{1,3}\\
& &  \verteq & \shortparallel & \verteq & \downineq & \cdots & \verteq\\
\cdots \\
\b^{1,N-1} & \to & 1 & b_2^{1,N-1} & b_3^{1,N-1} & b_4^{1,N-1} & \cdots & b_N^{1,N-1}\\
& &  \verteq & \verteq & \verteq & \verteq & \cdots & \downineq \\
\b^{1,N} & \to & 1 & b_2^{1,N} & b_3^{1,N} & b_4^{1,N} & \cdots & b_N^{1,N}\\
& &  \verteq & \verteq & \verteq & \verteq & \cdots & \verteq\\
\end{matrix}
\right.
\]
\[
\text{group 2 }
\left\{
\begin{matrix}
\b^{2,1} & \to & 1 & b_2^{2,1} & b_3^{2,1} & b_4^{2,1} & \cdots & b_N^{2,1}\\
& &  \verteq & \downineq & \verteq & \verteq & \cdots & \verteq\\
\b^{2,2} & \to & 1 & b_2^{2,2} & b_3^{2,2} & b_4^{2,2} & \cdots & b_N^{2,2}\\
& &  \verteq & \verteq & \downineq & \verteq & \cdots & \verteq\\
\b^{2,3} & \to & 1 & b_2^{2,3} & b_3^{2,3} & b_4^{2,3} & \cdots &  b_N^{2,3}\\
& &  \verteq & \shortparallel & \verteq & \downineq & \cdots & \verteq\\
\cdots \\
\b^{2,N-1} & \to & 1 & b_2^{2,N-1} & b_3^{2,N-1} & b_4^{2,N-1} & \cdots & b_N^{1,N-1}\\
& &  \verteq & \verteq & \verteq & \verteq & \cdots & \downineq \\
\b^{2,N} & \to & 1 & b_2^{2,N} & b_3^{2,N} & b_4^{2,N} & \cdots & b_N^{2,N}\\
& &  \verteq & \verteq & \verteq & \verteq & \cdots & \verteq\\
\end{matrix}
\right.
\]
\[
\text{group 3 }
\left\{
\begin{matrix}
\b^{3,1} & \to & 1 & b_2^{3,1} & b_3^{3,1} & b_4^{3,1} & \cdots & b_N^{3,1}\\
& &  \verteq & \downineq & \verteq & \verteq & \cdots & \verteq\\
\b^{3,2} & \to & 1 & b_2^{3,2} & b_3^{3,2} & b_4^{3,2} & \cdots & b_N^{3,2}\\
& &  \verteq & \verteq & \downineq & \verteq & \cdots & \verteq\\
\b^{3,3} & \to & 1 & b_2^{3,3} & b_3^{3,3} & b_4^{3,3} & \cdots &  b_N^{3,3}\\
& &  \verteq & \shortparallel & \verteq & \downineq & \cdots & \verteq\\
\cdots \\
\b^{3,N-1} & \to & 1 & b_2^{3,N-1} & b_3^{3,N-1} & b_4^{3,N-1} & \cdots & b_N^{1,N-1}\\
& &  \verteq & \verteq & \verteq & \verteq & \cdots & \downineq \\
\b^{3,N} & \to & 1 & b_2^{3,N} & b_3^{3,N} & b_4^{3,N} & \cdots & b_N^{3,N}\\
& &  \verteq & \verteq & \verteq & \verteq & \cdots & \verteq\\
\end{matrix}
\right.
\]
and continuing in this way we get an infinite matrix having $N$ columns.
With the notation $b_k^{i,j}$ we have that $i$=group, $j $=vector in the group, and $k$= the component.
We have
\[
b_{j+1}^{i,j}\geq b_{j+1}^{i,j+1},\,\text{ for $j=1,\cdots ,N-1$, and $i=1,2,\cdots $}
\]
and
\[
b_\ell^{i,j}=b_\ell^{i,j+1},\,\text{ for $\ell\neq j+1$}.
\]
We now look at each of the $N$ columns of the infinite matrix above. Each column has entries in non increasing order (the first column is obviously one), therefore the limit of the entries exists and is a number different from zero because the vectors belong to $W$ and therefore each limiting coordinate is bigger than $1/(1+\kappa)$.
Let $b^\infty_j$ be the limit of the entries in the column $j$, $j\geq 2$. Then the vector 
\[
\b^\infty=(1,b^\infty_2,b^\infty_2,\cdots , b^\infty_N)
\]
satisfies 
\begin{equation}\label{eq:limitinginequalityforbinfty}
f_j-\delta \leq \int_{\t_{\r(\b^\infty)}(m_j)}g(x)\,dx\leq f_j+\delta, \,j=2,\cdots ,N.
\end{equation}
In fact, fix $2\leq j\leq N$, the vector $\b^\infty$ is the limit of the vectors 
$\b^{i,j}$ as $i\to \infty$. But the vectors $\b^{i,j}$ verify 
\[
f_j-\delta \leq \int_{\t_{\r(\b^{i,j})}(m_j)}g(x)\,dx\leq f_j+\delta, \text{ for $i=1,2,\cdots $}.
\]
Since the function $\int_{\t_{\r(\b)}(m_j)}g(x)\,dx$ is continuous as a function of $\b$ for each $j$, Lemma \ref{lm:continuity}ii, taking the limit as $i\to \infty$ we obtain \eqref{eq:limitinginequalityforbinfty}. As it was shown before, the validity of \eqref{eq:limitinginequalityforbinfty} for $j\neq 1$ implies that \eqref{eq:limitinginequalityforbinfty} holds with $j=1$ and with $\delta$ replaced by $N\delta$.

\section{A Lipschitz estimate of $G_i$}\label{sec:lipschitzestimates}
Consider the map $G:\R^N_+ \to \R^N_{\geq 0}$ given by
\begin{equation}\label{defn of G}
G: \b = (b_1, \ldots, b_N) \to (G_1(\b), \ldots, G_N(\b))
\end{equation} 
where 
\begin{equation}\label{defn of G_i}
G_j(\b) = G_{\r(\b)}(m_j)
\end{equation}
for $j = 1, \ldots, N$ and $\R^N_+ = \{\b = (b_1, \ldots, b_k) : b_j > 0 \, \textrm{for} \, j = 1, \ldots, N\}. $

Let $\e_i$ be the unit vector in $\R^N$ with $1$ at the $i$-th position.
We shall compute $G_i(\b^t) - G_i(\b)$ where $\b^t = (b_1^t, \ldots, b_N^t) := \b+t \,\e_i.$ 
From Remark \ref{rmk:formulatauintermsofVj}
\[
\t_{\r(\b)}(m_i) = \Omega\cap \bigcap_{j=1}^N V_j
\]
except possibly on a set of measure zero,  with
\begin{equation}\label{the set V_j}
V_j = \left\{x \in S^2: \dfrac{b_i}{1-\kappa m_i\cdot x} \leq \dfrac{b_j}{1-\kappa m_j\cdot x}\right\},
\end{equation}
where for brevity we have used the notation $V_j$ for $V_{ij}$.
Likewise
\[
\t_{\r(\b^t)}(m_i) = \Omega\cap \bigcap_{j=1}^N V_j^t.
\]
where 
\begin{equation}\label{the set V_jt}
V_j^t = \left\{x \in S^2: \dfrac{b_i^t}{1-\kappa m_i\cdot x} \leq \dfrac{b_j^t}{1-\kappa m_j\cdot x}\right\}.
\end{equation}
We have $V_j^t=V_j=S^2$ for $j=i$. 
So
\begin{equation}\label{eq:formulasfortauRjneqi}
\t_{\r(\b)}(m_i) = \Omega\cap \bigcap_{j\neq i} V_j,\qquad 
\t_{\r(\b^t)}(m_i) = \Omega\cap \bigcap_{j\neq i} V_j^t.
\end{equation}

We prove the following proposition needed to show in Section \ref{sec:lipschitzestimateimpliesstop} that the algorithm stops in a finite number of steps.
\begin{proposition}\label{prop:estimateofGiisimpler} 
If $g$ is bounded in $\Omega$, then 
\begin{align}\label{eq:estimateofGiisimpler}
0\leq G_i(\b+t\,\e_i) - G_i(\b)   
&\leq \left(\sup_\Omega g\right) \sum_{r\neq i}\dfrac{C(\kappa,m_i\cdot m_r)}{b_r}\,(-t),
\end{align} 
for $-b_i<t<0$ and for each $\b\in \R^N_+$, where the constant $C(\kappa,m_i\cdot m_r)$ depends only on $\kappa$ and the angle between $m_i$ and $m_r$.
\end{proposition}
\begin{proof}
We have
\begin{align*}
 \text{$V_j^t\subset V_j$ for $t>0, j\neq i$ and $V_j\subset V_j^t$ for $t<0, j\neq i$,}
\end{align*}
so from \eqref{eq:formulasfortauRjneqi}
\begin{align*}
\t_{\r(\b^t)}(m_i)\subset \t_{\r(\b)}(m_i) \text{ for $t>0$}
\end{align*}
and
\begin{align*}
\t_{\r(\b)}(m_i)\subset \t_{\r(\b^t)}(m_i)  \text{ for $t<0$}.
\end{align*}
Since
\[
G_i(\b^t) - G_i(\b) = \int_{\t_{\r(\b^t)}(m_i)} g(x)dx - \int_{\t_{\r(\b)}(m_i)} g(x)dx,
\]
we obtain
\begin{equation*}
G_i(\b^t) - G_i(\b)
=
\begin{cases}
-\int_{\t_{\r(\b)}(m_i)\setminus \t_{\r(\b^t)}(m_i)} g(x)dx &\text{if $t>0$}\\
\int_{\t_{\r(\b^t)}(m_i)\setminus \t_{\r(\b)}(m_i)} g(x)dx &\text{if $t<0$.}
\end{cases}
\end{equation*}
If $t<0$, then we have
\begin{align*}
\t_{\r(\b^t)}(m_i)\setminus \t_{\r(\b)}(m_i)
&=
\Omega\cap \left\{ \left\{\cap_{j\neq i} V_j^t\right\}\setminus \left\{\cap_{r\neq i} V_r \right\}\right\}\\
&=
\Omega\cap \left\{\left\{ \cap_{j\neq i} V_j^t\right\}\cap 
\left\{\left( \cap_{r\neq i} V_r\right)^c \right\}\right\}\\
&=
\Omega\cap \left\{\left\{ \cap_{j\neq i} V_j^t\right\}\cap 
\left\{ \cup_{r\neq i} V_r^c \right\}\right\}\\
&= 
\Omega\cap \left\{ \cup_{r\neq i}\left\{V_r^c\cap \left\{ \cap_{j\neq i} V_j^t\right\} 
  \right\}\right\}\\
  &\subset \Omega\cap \left\{ \cup_{r\neq i}\left\{V_r^c\cap V_r^t
  \right\}\right\}\\
 &\subset  \cup_{r\neq i}\left(V_r^t\setminus V_r 
  \right) .
\end{align*}
On the other hand, if $t>0$, then
\[
\t_{\r(\b)}(m_i)\setminus \t_{\r(\b^t)}(m_i)\subset \cup_{r\neq i}\left(V_r\setminus V_r^t  
  \right). 
\] 
We will estimate for $-b_i<t<0$
\begin{align}\label{eq:estimateofGiwithsumofareas}
0\leq G_i(\b^t) - G_i(\b) &=\int_{\t_{\r(\b^t)}(m_i)\setminus \t_{\r(\b)}(m_i) } g(x)dx\leq   
\int_{\Omega\cap \left\{ \cup_{r\neq i}\left\{V_r^c\cap V_r^t
  \right\}\right\}}g(x)\,dx\notag\\
  &\leq \left(\sup_\Omega g\right) \text{area}\left(\cup_{r\neq i}\left(V_r^t\setminus V_r 
  \right) \right)\notag\\
  &\leq \left(\sup_\Omega g\right) \sum_{r\neq i}\text{area}\left(V_r^t\setminus V_r 
  \right).
\end{align}
We will calculate the area of $V_r^t\setminus V_r$ for $r\neq i$ and for $-b_i<t<0$.
\newline
{\bf Case $V_r= S^2$.} 
\newline
In this case, $V_r^t=S^2$ and so $\text{area}\left(V_r^t\setminus V_r 
  \right)=0$. 
\newline
{\bf Case $V_r\neq \emptyset$.} 
\newline
If $t\to -b_i$, then $V_r^t\to S^2$.
We will estimate the area measure of $V_r^t\setminus V_r$ when $-b_i<t<0$.  
The center of $V_r$ is the point $A_r=\dfrac{b_i m_r - b_r m_i}{|b_i m_r - b_r m_i|}$.
Fix an arbitrary vector $u$ from which we are going to measure the angles $\theta$. 
Given $0\leq \theta\leq 2\pi$ consider the points $\gamma_r(\theta,s)$ along the geodesic originating from $A_r$ and forming an angle $\theta$ with the vector $u$; $s$ denotes geodesic arc length. 
The point $\gamma_r(\theta,s)$ is on the boundary of $V_r$ if and only if the parameter $s = \tau_r = \cos^{-1} \left(\dfrac{b_i - b_r}{\kappa|b_i m_r - b_r m_i|}\right)$. 
Since $V_r\subset V_r^t$, and so the geodesic curve $\gamma_r(\theta,s)$ must intersect the boundary of $V_r^t$ for a unique value of $s$ with $s\geq \tau_r$.
Let us denote this value of $s$ by 
\[
h_r(\theta,t),
\]
and so 
\[
\gamma_r(\theta,s)\in \partial V_r^t \text{ if and only if } s=h_r(\theta,t).
\]
Let us set 
\[
x_t=\gamma_r(\theta,h_r(\theta,t)).
\]
Since $\gamma_r(\theta,s)$ is a geodesic curve from the point $A_r$ to the point $x_t$, 
we have
\begin{equation*}
h_r(\theta,t)=\arccos \left( A_r\cdot x_t\right).
\end{equation*}
On the other hand, the boundary of $V_r^t$ is the collection of points where the ellipsoids $E(m_i,b_i+t)$ and $E(m_r,b_r)$ intersect. So $x_t$ satisfies 
\begin{equation*}
\dfrac{b_i+t}{1 - \kappa x_t \cdot m_i} = \dfrac{b_r}{1 - \kappa x_t \cdot m_r},
\end{equation*}
which yields
\[
b_r(1-\kappa\,x_t\cdot m_i)=(b_i+t)(1-\kappa \,x_t\cdot m_r)
\]
which using the definition of $A_r$ yields
\[
A_r\cdot x_t=\dfrac{b_i-b_r}{\kappa \,|b_im_r-b_rm_i|}
+
\dfrac{1-\kappa\,x_t\cdot m_r}{\kappa \,|b_im_r-b_rm_i|}\,t
=
\cos s+\dfrac{1-\kappa\,x_t\cdot m_r}{\kappa \,|b_im_r-b_rm_i|}\,t,
\]
where in the last identity we used the definition of $s = \tau_r$.
We are now ready to calculate the surface area of $V_r^t\setminus V_r$.
Integrating in polar coordinates we obtain 
\begin{align}\label{eq:estimateofareasofgeodesicdisksr}
\text{area}(V_r^t\setminus V_r)
&= \int_0^{2\pi}\int_{\tau_r}^{h_r(\theta,t)}\sin s\,ds\,d\theta\notag\\
&=\int_0^{2\pi} \left(\cos \tau_r-\cos h_r(\theta,t)\right)\,d\theta
=(-t)\,\int_0^{2\pi} \dfrac{1-\kappa\,x_t\cdot m_r}{\kappa \,|b_im_r-b_rm_i|}\,d\theta\notag\\
&\leq (-t)\,2\pi\,\dfrac{1+\kappa}{\kappa}\,\dfrac{1}{|b_im_r-b_rm_i|}\leq C(\kappa,m_i\cdot m_r)\,\dfrac{1}{\max\{b_r,b_i\}}\,(-t),
\end{align}
\footnote{Since $m_i\neq m_r$ and have absolute value one, we have $m_i\cdot m_r\leq 1-\delta$ for some $1>\delta>0$. We then have 
$|b_im_r-b_rm_i|^2=b_r^2 - 2 b_rb_i m_r\cdot m_i + b_i^2\geq b_r^2 - 2 b_rb_i (1-\delta) + b_i^2= (b_r-(1-\delta)b_i)^2+b_i^2\delta(2-\delta)\geq b_i^2\delta(2-\delta)$.
Similarly, $|b_im_r-b_rm_i|^2\geq b_r^2\delta(2-\delta)$.}
for $-b_i<t<0$, where $C(\kappa,m_i\cdot m_r)$ is a positive constant depending only on $\kappa$ and the 
dot product $m_i\cdot m_r$.
\newline
{\bf Case when $V_r=\emptyset$}.
\newline
Let us recall that 
\[
V_r=\left\{x \in S^2: \dfrac{b_i}{1-\kappa m_i\cdot x} \leq \dfrac{b_r}{1-\kappa m_r\cdot x}\right\}
\] 
and
\[
V_r^t=\left\{x \in S^2: \dfrac{b_i+t}{1-\kappa m_i\cdot x} \leq \dfrac{b_r}{1-\kappa m_r\cdot x}\right\}.
\]
We have that 
\[
V_r^t=\left\{x\in S^2:x\cdot \dfrac{(b_i+t) m_r - b_r m_i}{|(b_i+t) m_r - b_r m_i|}\geq \dfrac{b_i+t - b_r}{\kappa \,|(b_i+t) m_r - b_r m_i |} \right\}
\]
Let
\begin{equation}\label{def:functiongt}
g(t)=\dfrac{b_i+t - b_r}{\kappa \,|(b_i+t) m_r - b_r m_i |}, \,t\in (-b_i,+\infty).
\end{equation}
Since $V_r=\emptyset$ we have
\[
g(0)>1.
\]
If we set $\Delta(t)=|(b_i+t) m_r - b_r m_i |$, then by calculation
\[
g'(t)=\dfrac{b_r(b_i+t+b_r)(1-m_i\cdot m_r)}{\kappa\,\Delta(t)^3},
\]
and therefore 
\[
g'(t)>0,\qquad \forall t\in (-b_i,+\infty).
\]
Also
\[
g(t)\to -\dfrac{1}{\kappa}, \text{ when $t\to -b_i$},
\]
and
\[
g(t)\to \dfrac{1}{\kappa}, \text{ when $t\to \infty$}.
\]
Therefore there is a unique number $-b_i<t_0<0$ such that $V_r^t=\emptyset $ for $t_0<t<0$ and 
$V_r^t\neq \emptyset $ for $-b_i<t\leq t_0$; this is the value for which $g(t_0)=1$. In particular, when $t=t_0$, the set $V_r^{t_0}$ consists only of the center point
$A_{r,t_0}=\dfrac{(b_i+t_0) m_r-b_r m_i}{|(b_i+t_0) m_r-b_r m_i|}$.
We need to calculate the area of $V_r^t\setminus V_r=V_r^t$ for $-b_i<t\leq t_0$.
To do this we will use the calculation from the previous case with $V_r\leadsto V_r^{t_0}$. 
In fact, we now parametrize the boundary of $V_r^t$ from the center of $V_r^{t_0}$, $A_{r,t_0}$.
Fix a arbitrary vector $u$ from which we are going to measure the angles $\theta$. 
Given $0\leq \theta\leq 2\pi$ consider the points $\gamma_r(\theta,s)$ along the geodesic originating from $A_{r,t_0}$ and forming an angle $\theta$ with the vector $u$; $s$ denotes geodesic arc length. 
The point $\gamma_r(\theta,s)$ is on the boundary of $V_r^{t_0}$ if and only if the parameter $s = \tau_{r,t_0} = \cos^{-1} \left(\dfrac{(b_i+t_0)-b_r}{\kappa\,|(b_i+t_0) m_r-b_r m_i|}\right)=0$. 
Since $t\leq t_0$, $V_r^{t_0}\subset V_r^t$, and so the geodesic curve $\gamma_r(\theta,s)$ must intersect the boundary of $V_r^t$ for a unique value of $s$ with $s> \tau_{r,t_0}=0$.
Let us denote this value of $s$ by 
\[
h_r(\theta,t),
\]
and so 
\[
\gamma_r(\theta,s)\in \partial V_r^t \text{ if and only if } s=h_r(\theta,t).
\]
Let us set 
\[
x_t=\gamma_r(\theta,h_r(\theta,t)).
\]
Since $\gamma_r(\theta,s)$ is a geodesic curve from the point $A_{r,t_0}$ to the point $x_t$, 
we have
\begin{equation*}
h_r(\theta,t)=\arccos \left( A_{r,t_0}\cdot x_t\right).
\end{equation*}
On the other hand, the boundary of $V_r^t$ is the collection of points where the ellipsoids $E(m_i,b_i+t)$ and $E(m_r,b_r)$ intersect. So $x_t$ satisfies 
\begin{equation*}
\dfrac{b_i+t}{1 - \kappa x_t \cdot m_i} = \dfrac{b_r}{1 - \kappa x_t \cdot m_r},
\end{equation*}
which yields
\[
b_r(1-\kappa\,x_t\cdot m_i)=(b_i+t)(1-\kappa \,x_t\cdot m_r)=(b_i+t_0)(1-\kappa \,x_t\cdot m_r)+
(t-t_0)(1-\kappa \,x_t\cdot m_r)\]
which using the definition of $A_{r,t_0}$ yields
\begin{align*}
A_{r,t_0}\cdot x_t&=\dfrac{(b_i+t_0)-b_r}{\kappa \,|(b_i+t_0) m_r-b_r m_i|}
+
\dfrac{1-\kappa\,x_t\cdot m_r}{\kappa \,|(b_i+t_0) m_r-b_r m_i|}\,(t-t_0)\\
&=
1+\dfrac{1-\kappa\,x_t\cdot m_r}{\kappa \,|(b_i+t_0) m_r-b_r m_i|}\,(t-t_0),
\end{align*}
where in the last identity we used that $\tau_{r,t_0}=0$.
Integrating in polar coordinates as before we obtain for $-b_i<t\leq t_0$ that
\begin{align*}
\text{area}(V_r^t)
&= \int_0^{2\pi}\int_{0}^{h_r(\theta,t)}\sin s\,ds\,d\theta\\
&=\int_0^{2\pi} \left(1-\cos h_r(\theta,t)\right)\,d\theta
=(t_0-t)\,\int_0^{2\pi} \dfrac{1-\kappa\,x_t\cdot m_r}{\kappa \,|(b_i+t_0) m_r-b_r m_i|}\,d\theta\\
&\leq (t_0-t)\,2\pi\,\dfrac{1+\kappa}{\kappa}\,\dfrac{1}{|(b_i+t_0) m_r-b_r m_i|}\\
&\leq C(\kappa,m_i\cdot m_r)\,\dfrac{1}{\max\{b_i+t_0,b_r\}}\,(t_0-t)\\
&\leq
C(\kappa,m_i\cdot m_r)\,\dfrac{1}{b_r}\,(t_0-t)\leq \dfrac{C(\kappa,m_i\cdot m_r)}{b_r}\,(-t), 
\end{align*}
since $t_0<0$, 
where $C(\kappa,m_i\cdot m_r)$ is a positive constant depending only on $\kappa$ and the 
dot product $m_i\cdot m_r$.
\newline
As a conclusion we obtain combining all cases the proposition.
\end{proof}

\begin{remark}\rm
Similar estimates for $G_i(\b)$ hold when the increment are in the variables $b_r$ with $r\neq i$.
In fact, using arguments similar to the ones used in the proof of the last proposition one can show that
\[
0\leq G_i(\b+t\,\e_r) - G_i(\b)  \leq C_\kappa \,\sup_\Omega g\,\dfrac{1}{\max\{b_i,b_r\}} \,t
\]
for all $0<t<\infty$ and for each $\b\in \R^N_+$, $r\neq i$.
Using these estimates for $r\neq i$, and the fact that in the far field the refractor measure is invariant by dilations, one can also obtain the estimate \eqref{eq:estimateofGiisimpler}.
\end{remark}

\section{Numerical analysis}\label{sec:numericalanalysis}
In order to see our algorithm in action, we implemented routines in the C/C++ programming language to produce some concrete numerical examples of refractors for a given
output image.
\footnote{All software used in our numerical investigation and graphical results can be
found at \href{http://helios.physics.howard.edu/~deleo/Refractor/}{http://helios.physics.howard.edu/~deleo/Refractor/}}.
We will assume that the function $g$ in Definition \ref{sol1} is constant.
\begin{figure}
  \centering
  \includegraphics[width=12cm]{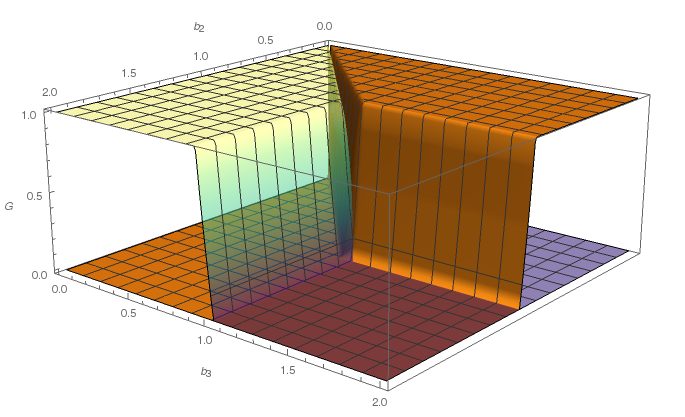}
  \caption{%
    \small
    Graph of the functions $G_2(\b)=G_{\r(\b)}(m_2)$ (semitransparent) and $G_3(\b)=G_{\r(\b)}(m_3)$ (opaque) in the $[0,2]^2$
    square for the case with three unit output directions $m_1,m_2,m_3$ given by the homogeneous coordinates $[0:0:1]$, $[0:1:5]$,
    and $[1:0:5]$, respectively.
}\label{fig:graph}
\end{figure}

Assuming $b_1=1$ and conservation of energy $G_1(1,b_2,\cdots ,b_N)+G_2(1,b_2,\cdots ,b_N)+\cdots +G_N(1,b_2,\cdots ,b_N)=$constant, 
we have from Remark \ref{rmk:Gconstantoutsideasetdefinedbylinearinequalities} that the map 
\begin{equation}\label{eq:mapfromN-1toN-1}
(b_2,\cdots ,b_N)\mapsto \left(G_2(1,b_2,\cdots ,b_N),\cdots ,G_N(1,b_2,\cdots ,b_N)\right)
\end{equation}
has a highly degenerate Jacobian in
a large region of the phase space (e.g. see Figure \ref{fig:graph}), that is, in the region $\cup_{i=1}^NF_i$ (with $b_1=1$). 
Notice that the vector $\b$ in \eqref{conditionfornumericalsolution} belongs to $\left(\cup_{i=1}^NF_i\right)^c$.
%
%

To evaluate numerically the output intensities $G_i(\b)=G_{\r(\b)}(m_i)$ for any fixed $\b=(1,b_2,\cdots ,b_N)$, we proceed as follows. We discretize $\Omega$ into a finite array 
of directions $A$. 
Fix a direction $\gamma \in A$ and considered the ray, denoted by $\ell_{\gamma}$, from 0 having direction $\gamma$. 
Now all the ellipsoids $E(m_j,b_j)$ intersect the ray $\ell_{\gamma}$ at some point $P(j,\gamma)$. Then there is a $j_\gamma$ such that the distance from $P(j,\gamma)$ to the origin is minimum, and we choose this ellipsoid. So for each $\gamma \in A$ we have an index $j_\gamma$ such that the ellipsoid $E(m_{j_\gamma},b_{j_\gamma})$ intersects the ray $\ell_{\gamma}$ at the point having minimum distance to the origin.
Since the refractor is by definition the minimum of the polar radii of ellipsoids, then, in the direction $\gamma$, the refractor refracts into the direction $m_{j_\gamma}$.
This way we have a map $T$ from each $\gamma \in A$ into a vector $m_{j_\gamma}$.
Clearly, this map $T$ might not cover all of the $m_1,...,m_N$.
We have 
\begin{equation}\label{eq:numericaldefinitionofGi}
G_i(\b)=\dfrac{\#\{\gamma\in A: T(\gamma)=m_i\}}{\#\{\gamma \in A\}}.
\end{equation}
To reduce computational time in the calculation of $G_i(\b)$, it is helpful not only to keep track of how many of the directions $\gamma$ get refracted in the direction $m_i$, 
but also to record $T(\gamma)$ at each $\gamma\in A$.
This is because in two consecutive steps of the algorithm described in Section \ref{sec:algorithm} we need to compute the values of  
$G_i(\b)$ and $G_i(\b')$, where $\b$ and $\b'$ are two vectors that differ only in one component.
In fact, suppose $\b$ and $\b'$ are successive values in the algorithm differing only in the $j_0$-th component, and we know $T(\gamma) = m_j$ relative to  $\b$. 
To evaluate $T(\gamma)$ relative to $\b'$ and subsequently obtaining the value of $G_i(\b')$, we only need to consider the ellipsoids $E(m_{j_\gamma},b_{j_\gamma})$ and $E(m_{j_0},b'_{j_0})$ and the distance from the origin to  $P(j, \gamma)$ and $P(j_0, \gamma).$\footnote{First notice that $T$ depends on $\b$. If $\b'$ and $\b$ are as in Lemma \ref{monotonicityoftrace}, then if $\gamma \in \t_{\b}(m_j)$, then 
$\gamma \in \t_{\b'}(m_\ell)$ or $\gamma \in \t_{\b'}(m_j)$ ($\gamma$ no singular). Because if 
$\gamma \not\in \t_{\b'}(m_\ell)$, then $\gamma \in \t_{\b'}(m_k)$ for some $k\neq \ell$. Then by \eqref{eq:inclusionwithinotell}, 
$\gamma\in \t_{\b}(m_k)$, and since $\gamma$ is not singular, we get $k=j$.} 
By doing so we cut the running time by a factor of $N$.

From \eqref{eq:numberoftotaliterations}, we know that the number of iterations needed to find the optimal vector $\b$, for which $err=\max_{2\leq i\leq N}|f_i-G_i(\b)|<\delta$, grows not faster than $N^2/\delta$. We expect it not to grow slower than this as well, so that we expect a theoretical computational time of order $O(N^2/\delta)$.
In addition, to use smaller values of $\delta$ requires increasing the value of $K$
and therefore increasing also the number of directions in $A$ to test (see the end of Section \ref{subsect:detaileddescriptionofthealgorithm}).
Indeed, for any given $A$, from \eqref{eq:numericaldefinitionofGi} the set of values that $G_i(\b)$ takes on is finite. Therefore for $\delta$ small enough there is $j_0$ such that we cannot find a value of $b_{j_0}$ for which $f_{j_0}< G_{j_0}(\b) < f_{j_0}+\delta$. This means that, if we want to find a $\b$ such that $err<\delta$, we need to increase the size of $A$ so that $\#A>1/\delta$. Since the loop on $A$ leads to a running time proportional to $\# A$, in our implementation we expect a computational time of order $O(N^2/\delta^2)$.

For the calculations here we choose $\Omega$ as the intersection of the upper semi sphere in $\R^3$ with 
the cone with vertex at the origin and generated by the vectors 
$(1,1,2),(-1,1,2),(-1,-1,2)$ and $(1,-1,2)$.
The set  $\Omega^*$ is the intersection of the upper semi sphere in $\R^3$ with 
the cone with vertex at the origin and generated by the vectors 
$(1,1,5),(-1,1,5),(-1,-1,5)$ and $(1,-1,5)$.
This choice of the domains $\Omega$ and $\Omega^*$ satisfy the condition~\eqref{geometricconstraintfork<1} avoiding total internal reflection
when $\kappa=1/2$.
Inside $\Omega^*$, we choose the refracted directions $\{m_i\}_{1\leq i\leq N}$, with $N=(n+1)^2$, as 
$$
\Omega^*_N=\left\{[r:r':5n]: r,r'=-n,-n+2,\dots,n-2,n;\text{ with $r,r'$ integers}\right\};
$$
where $[r:r':2n]$ denotes the unit vector in the direction $(r,r',2n)$.
We discretize $\Omega$ into $K=(2M+1)^2$ points having the form 
$$
\Omega_K=\left\{[r:r':2M]:-M\leq r,r'\leq M;\text{ with $r,r'$ integers}\right\}.
$$

We always start the algorithm in Section~\ref{sec:algorithm} with a vector in $W$, the set of admissible vectors, satisfying $b_1=1$ and $b_i=2$ for $i\geq 2$ to obtain a vector $\b$ satisfying \eqref{conditionfornumericalsolution}, with $\epsilon=1/10N$, and uniform output intensities $f_i=1/N$, $1\leq i \leq N$, for the directions $m_i\in\Omega^*_N$. 
That is, we stop our computations when
$$
\max_{1\leq i\leq N}\left|G_{\r(\b)}(m_i)-\dfrac{1}{N}\right|\leq\epsilon=\dfrac{1}{10N}.
$$
While implementing the algorithm for $1\leq n\leq 10$,
with $\delta=\epsilon/N=1/(10N^2)$, as in
Section~\ref{subsect:detaileddescriptionofthealgorithm}, 
our data in Fig.~\ref{fig:growth}a, show that the number of iterations $\nu$ grows roughly as $\nu(N)\simeq 0.3N^{2.8}$; although the exponent appears 
to slow down towards 2 as $N$ increases. 
This is consistent with~(\ref{eq:numberoftotaliterations}), according to which the growth cannot be faster than $N^4$. Similarly, for the running times 
$\tau$ we observe that $\tau(N)\simeq0.003 N^3$. Note that the evident jump in the running times when $n\geq 7$ 
is due to the fact that the values of $\delta$ for these cases get so small that for the algorithm to complete successfully it is necessary to use for these cases larger values of $K$ 
($M=200$ for $n$ up to 4, $M=250$ for $5\leq n\leq6$, $M=500$ for $n=7$, $M=350$ for $n=8$, $M=800$ for $n=9$ and $M=1100$ for $n=10$).
%
\begin{figure}
\centering
    \begin{tabular}{cc}
      \includegraphics[width=6.cm]{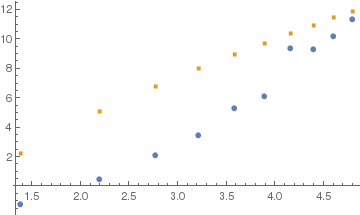}
      &
      \includegraphics[width=6.cm]{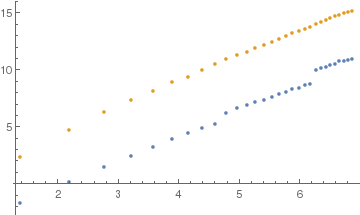}\cr
      \tiny a) Runtime/iteration time growth with full error&\tiny b) Runtime/iteration time growth with partial error\cr
      \includegraphics[width=5.cm]{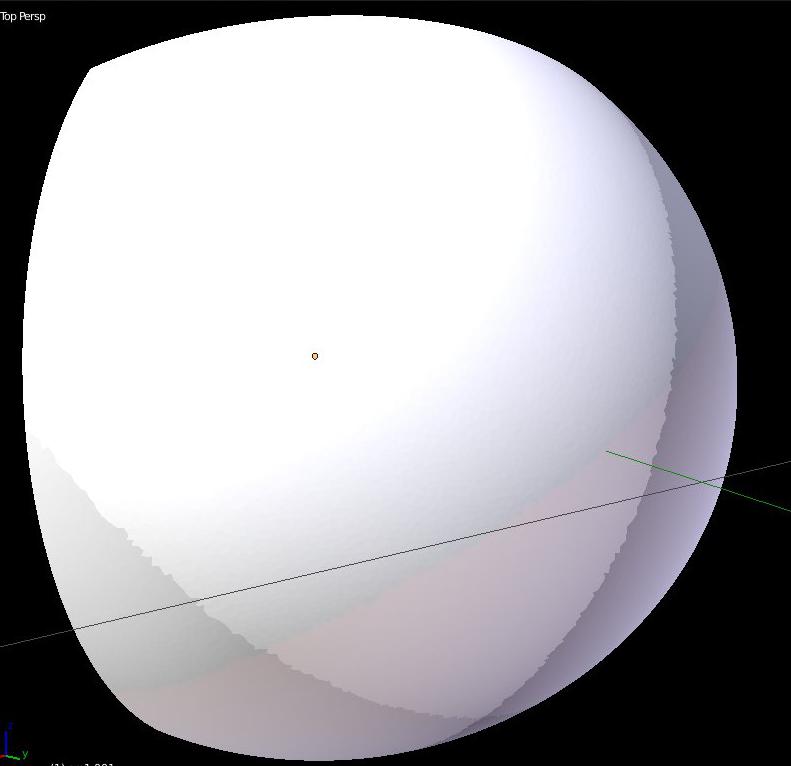}
      &
      \includegraphics[width=6.cm]{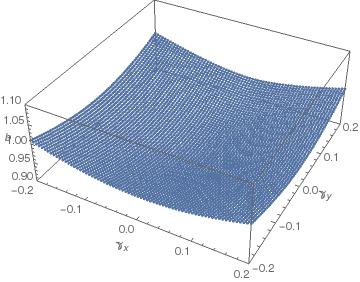}\cr
      \tiny c) Lens &\tiny d) Vector $\b$ for $N=5041$\cr
    \end{tabular}\\
  \caption{%
    \small
    a) Growth of runtime and number of iterations in our implementation of the algorithm of Section~\ref{subsect:detaileddescriptionofthealgorithm}
    when we minimize $|G_{\r(\b)}(m_i)-1/N|$ for {\sl all} $m_i$. Blue is $(\ln N, \ln \tau(N))$; orange is $(\ln N, \ln \nu(N))$. 
    b) Same plot when we disregard what happens in the direction $m_1$,
    as a function of the number of output directions. 
    c) Detail of the lens giving rise to Descartes' image with $N=71^2$ output directions. 
    d) 
    Plot of the components of the vector $\b$, considered as a map $\Omega^*_N\to\R$, when the refractor $\r(\b)$ gives the $71\times 71$ Descartes' picture in Figure \ref{fig:descartes} b. The set $\Omega^*_N$ is represented as the array of points $(r/5n,r'/5n)$, $r,r'=-n,-n+2,\dots,n$, inside the square $[-1/5,1/5]^2$.
}\label{fig:growth}
\end{figure}
Such a fast growth suggests that, although the algorithm in Section~\ref{sec:algorithm} always yields a solution $\b$ after a finite number of iterations, it might take a long computational time for large values of $n$.
For example for $n=30$, namely $N=961$, these data predict
a running time of at least 34 days.

The running time decreases considerably if we disregard the direction $m_1$.
In fact, in order to be able to use the algorithm in Section~\ref{subsect:detaileddescriptionofthealgorithm} with higher values of $N$, we disregard the intensity in the first refracted direction
$m_1$, namely, we stop our computations when 
$$
\max_{2\leq i\leq N}\left|G_{\r(\b)}(m_i)-\dfrac{1}{N}\right|\leq \epsilon=\dfrac{1}{10N}.
$$
As it is clear from the discussion in Section~\ref{subsect:detaileddescriptionofthealgorithm}, in order to achieve this result it is enough to 
take $\delta\leq \epsilon$. 
This way, omitting $m_1$, $\delta$ will decrease more slowly with $N$ and, accordingly, the number of iterations will grow slower with $N$ and with the 
size of the discretization of $\Omega$. Therefore, the running time will be shorter. In Fig.~\ref{fig:growth}b we show the growth of the number of iterations 
and running time when $1\leq n\leq30$, corresponding to $4\leq N\leq 961$.
In this case, the data show a growth in the number of iterations $\nu$ roughly quadratic in the number 
$N$ of the output directions: $\nu(N)\simeq 2.7N^{2.05}$. Similarly, for the running times $\tau$, we observe that $\tau(N)\simeq\alpha N^{1.9}$.
Here $\alpha$ depends on the value of $\delta$, and so from the duration of every single step in the program's loop to evaluate the map $T(\gamma)$,
the size of the discretization $\Omega_K$ (and therefore the number of steps in the loop above), as well as on non mathematical factors like
the hardware on which the program runs\footnote{All data in Fig.~\ref{fig:growth} and Fig~\ref{fig:descartes} are produced on an Intel Xeon 2.6GHz CPU} 
and the coding details of the algorithm implementation. For $1\leq n\leq9$ (see Fig.~\ref{fig:growth}a) we use
$\delta=10^{-3}$ and $M=200$ and find $\alpha\simeq 0.03$ seconds. For $n\geq 10$, the value $\delta=10^{-3}$ is not small enough
for the algorithm to reach a $10\%$ error and so we lower it to $\delta=2\cdot10^{-4}$. This change of course increases $\alpha$, leading to the visible jump in the (log-log) 
graph of $\tau(N)$. For $10\leq n\leq 22$ we find $\alpha\simeq0.05s$.
For $n\geq 23$, a discretization of $\Omega$ with $M=200$ is not fine enough to allow the evaluation 
of the map $T(\gamma)$. So we increase $M$ to 300, leading to a second visible jump in the graph corresponding to the larger value $\alpha\simeq 0.135$  secs.
For example, with this last value of $\alpha$, we get a lower bound of about 16 days for the running time in the case $N=5000$, i.e, $n\approx 70$. 

The results can be obtained faster combining this algorithm for small values of $n$ with a quasi-Newtonian
root-finding algorithm.  Such methods are generalizations of the Newton method to find the root of a function 
without an explicit expression of its Hessian. 
The problem is that, as for the Newton method, quasi-Newtonian methods require a starting point where the function has a non-degenerate Jacobian
and, as we already pointed out at the beginning of the section, the function \eqref{eq:mapfromN-1toN-1} has a degenerate Jacobian in a large portion of its domain.
We use the GNU Scientific Library (GSL) implementation of the quasi-Newtonian version of Powell's 
Hybrid method, since this method does not need an explicit Jacobian. 

Therefore, as a first step we use the algorithm from Section~\ref{subsect:detaileddescriptionofthealgorithm} (disregarding the direction $m_1$) to find a vector $\tilde \b$ 
inside
the region where the Jacobian is non-degenerate. And next use $\tilde \b$ as a starting point of the quasi-Newtonian algorithm to find a vector $\b^*$ for which
the output intensities $G_i(\b^*)$ are ``close enough'' to the $f_i=1/N$.
In fact, we start by evaluating a vector $\tilde\b=(\tilde b_1,\dots,\tilde b_{961})$ which gives {\it homogeneous light intensity ($f_i=1/N$) in all directions (except $m_1$)} 
 within $10\%$ for the output array $\Omega^*_{961}$, corresponding to $n=30$. This computation, with $\delta=10^{-4}$ and $M=300$, 
took about 15 hours. 
The vector $\tilde\b$ is then used as starting point by any quasi-Newtonian method to 
find the desired vector $\b^*=(b_1,\dots,b_{961})$ such that $\max_{1\leq i\leq 961}\left|G_{\r(\b)}(m_i)-1/N\right|<\epsilon$ over the array 
$\Omega^*_{961}$ and any (reasonable) $\epsilon$.
With this method, it takes only about 25 minutes to find, starting from the vector $\tilde\b$,
a vector $\b^*$ giving rise to a homogeneous distribution of light ($f_i=1/N$) in {\sl all} the directions of the array $\Omega^*_{961}$ within $10\%$! 


We now use the vector $\b^*$ as a pivot in a concrete case; 
namely, to produce a lens that yields an image of a famous portrait of Descartes by
Frans Hals  on the array of refracted directions $\Omega^*_{14641}$, corresponding to $n=120$. 
The images produced with the lens using LuxRender are shown in Fig.~\ref{fig:descartes} for various resolutions. 
First of all, we need to
extract from a digital version of the original picture the output intensities $f_i$, $1\leq i\leq 14641$.
For this purpose we use Imlib2, a general 
purpose open source C library aimed at images manipulation. 
Our final goal is finding a vector $\b$ so that the refractor $\r(\b)$ satisfies the inequalities
\begin{equation}\label{eq:obtainvectorberror}
\max_{1\leq i\leq14641}\left|G_{\r(\b)}(m_i)-f_i\right|\leq\min_{1\leq i\leq 14641}\{f_i\}/10.
\end{equation} 
Note that, since the naked eye cannot
usually detect nuances of black within a complex picture, and since for large arrays the amount of light in dark spots is very low, it is actually enough for 
us that the max and min in \eqref{eq:obtainvectorberror} are taken only over $i$ such that $f_i$ is sufficiently large when $N$ is large ($f_i$ small corresponds with dark spots).
Heuristically for this particular case we set this number to be $30\%$ of 
the total number of refracted directions. 
%
\begin{figure}  \centering
    \begin{tabular}{cc}
      \includegraphics[width=4.2cm]{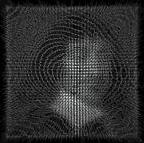}&\includegraphics[width=4.2cm]{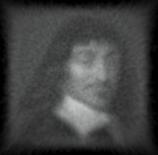}\cr
      \tiny a) Rendering (VTK 41x41)& \tiny b) Rendering (CGAL 71x71)\cr
      \includegraphics[width=4.2cm]{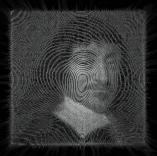}&\includegraphics[width=4.2cm]{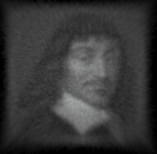}\cr
      \tiny c) Rendering (VTK 121x121)&\tiny d) Rendering (CGAL 121x121)\cr
    \end{tabular}\\
  \caption{%
    \small
    Rendering of Descartes' image from 3D models generated by the graphic libraries VTK and CGAL. The rendering 
    has been done via LuxRender, a physically accurate raytracer engine, through the modeling package Blender.
}\label{fig:descartes}
\end{figure}

Now we evaluate the coefficients
$f_i$ corresponding to Descartes' picture for the array $\Omega^*_{961}$. Next using $\b^*$, calculated in the first step, as a starting point in the quasi-Newtonian algorithm, we find the corresponding $\b$ giving rise to the $f_i$.
It takes 
about 23 minutes to get a $\b$ such that all $G_{\r(\b)}(m_i)$ are within $10\%$ from the $f_i$; all but one within 1\%, and 96\% of them are within .1\%.
At this point, we consider the array $\Omega^*_{1681}$, corresponding to $n=40$, evaluate the $f_i$'s corresponding to 
Descartes' picture on this array and use a standard interpolation algorithm (in concrete we use an implementation available in the GSL) 
to interpolate the values of $(b_1,\dots,b_{961})$ into a new vector $(\tilde b_1,\dots,\tilde b_{1681})$ and finally use this as starting point 
for the quasi-Newtonian code to find a vector $\b=(b_1,\dots,b_{1681})$ giving rise to the $f_i$, $1\leq i\leq 1681$, within $10\%$.
 It takes
about 28 minutes to find a $\b$ such that all $G_{\r(\b)}(m_i)$ but three are within $10\%$ from the corresponding $f_i$ (and $98\%$ of them is actually within $1\%$).
From this we move to the array $\Omega^*_{2601}$, interpolate the previous $\b=(b_1,\dots,b_{1681})$ to a new $\tilde\b=(b_1,\dots,b_{2601})$ and use it as a starting point
for the quasi-Newtonian algorithm, that in about 3 hours is able to find a $\b$ such that all $G_{\r(\b)}(m_i)$ but five are within $10\%$ from the corresponding $f_i$.
We continue with this process by increasing $n$ by 10 at every step until we arrive to $n=120$, which provides the final $\b$ (see Fig.~\ref{fig:descartes}c,d)
so that the $70\%$ of the $G_{\r(\b)}(m_i)$ are within $10\%$ from the corresponding $f_i$. The last computational step took about 2 days. The process can be continued to obtain higher resolution pictures.

\section{Conclusion}
We have obtained a numerical procedure to find far field refractors with arbitrary precision when the target is discrete composed of $N$ directions, and radiation emanates from one source point.
The density of the incoming radiation is assumed only bounded away from zero and infinity, and the domains $\Omega$ are general subsets of the unit sphere having boundary with surface measure zero.
The procedure converges in a finite number of steps and an estimate of this number is given in terms of $N$, the angles between the different directions in the target,  and the required approximation. To show the convergence we prove a Lipschitz estimate of the refractor map. 
A numerical implementation of the algorithm is carried out by using C/C++ programming language, and concrete examples of refractors for a given output image are provided.
The near field case can be treated with similar methods and we will return to this problem in the near future.

\providecommand{\bysame}{\leavevmode\hbox to3em{\hrulefill}\thinspace}
\providecommand{\MR}{\relax\ifhmode\unskip\space\fi MR }
\providecommand{\MRhref}[2]{%
  \href{http://www.ams.org/mathscinet-getitem?mr=#1}{#2}
}
\providecommand{\href}[2]{#2}

\end{document}